\documentclass[12pt]{article}
\usepackage[left=2.4cm,top=2.4cm,right=2.4cm,bottom=2.5cm]{geometry}

\usepackage[normalem]{ulem}

\usepackage[utf8]{inputenc}
\usepackage{amsmath}
\usepackage{amsthm}
\newtheorem{theorem}{Theorem}[section]
\usepackage{amsfonts}
\usepackage{amssymb}

\newtheorem{lemma}[theorem]{Lemma}
\newtheorem{corollary}[theorem]{Corollary}

\newenvironment{claim}[1]{\par\noindent\textit{Claim:}\space#1}{}
\newenvironment{claimproof}[1]{\par\noindent\textit{Proof:}\space#1}{\leavevmode\unskip\penalty9999 \hbox{}\nobreak\hfill\quad\hbox{$\diamond$}}

\usepackage[framemethod=TikZ]{mdframed}
\usepackage[colorinlistoftodos]{todonotes}
\title{Deduction with $k$ moves}

\author{Andrea C.~Burgess\thanks{Department of Mathematics and Statistics, University of New Brunswick, Saint John, NB, Canada.  \linebreak ORCID: 0009-0001-0504-7823.}  %Funding: NSERC (RGPIN-2025-04633)}
\thanks{Corresponding author.  Email: andrea.burgess@unb.ca.}
\and Nancy E.~Clarke\thanks{Department of Mathematics and Statistics, Acadia University, Wolfville, NS, Canada.\linebreak ORCID: 0000-0002-0597-1717.}
\and Shannon L.~Fitzpatrick\thanks{School of Mathematical and Computational Sciences, University of Prince Edward Island, Charlottetown, PE, Canada. ORCID: 0000-0001-6986-8965.}
\and
Melissa A.~Huggan\thanks{Department of Mathematics, Vancouver Island University, Nanaimo, BC, Canada.\linebreak ORCID: 0000-0001-7923-515X.}
}

\date{\today
}

\begin{document}

\maketitle

\begin{abstract}
The deduction game may be thought of as a variant on the classical game of cops and robber in which the cops (searchers) aim to capture an invisible robber (evader); each cop is allowed to move at most once, and cops situated on different vertices cannot communicate to co-ordinate their strategy. In this paper, we extend the deduction game to allow each searcher to make $k$ moves,
where $k$ is a fixed positive integer. We consider the value of the $k$-move deduction number on several classes of graphs including paths, cycles, complete graphs, complete bipartite graphs, and Cartesian and strong products of paths.
\medskip



\end{abstract}

\section{Introduction}

The {\em deduction game}, which was introduced in~\cite{BDF2024, Farahani2022}, may be thought of as a variant on the classical game of cops and robber. Cops and robber, first considered in~\cite{NW1983, Q1978}, is a two-player game on a graph, with one player controlling a set of cops and the other controlling a single robber.  The cops and robber are positioned on vertices, and play proceeds in alternating turns; on a turn, players may move their token(s) to an adjacent vertex or pass.  The aim of the cops is to capture the robber by occupying the robber's vertex, while the robber attempts to avoid capture indefinitely.  

While many variants of cops and robber have been proposed (see, for instance,~\cite[Chapter~8]{BonatoNowakowski}), deduction arose from variants which restrict the cops, particularly {\em zero-visibility cops and robber} (see~\cite{DDTY2015a, DDTY2015b, Tang2004, Tosic1985}), in which the cops cannot see the robber's location until the moment of capture, and {\em time-constrained} variants (see~\cite{ADHY2008}), in which the cops must achieve capture within a fixed number of turns in order to win. Deduction adds to these the restriction that cops situated on different vertices cannot communicate to co-ordinate their strategy; thus cops' movements depend on local conditions regarding the vertices in their immediate neighbourhood.  

In the deduction game as initially defined in~\cite{BDF2024,Farahani2022}, cops are termed {\em searchers}.  Each searcher is permitted to move at most once. In this paper, we extend the game to allow searchers $k$ moves, where $k$ is a fixed positive integer.  We view deduction as a multi-stage process.  In the initial set-up, searchers are placed on vertices of a graph.  The initial placement of the searchers is called a {\em layout}.  Note that in a layout, searchers need not be situated on distinct vertices.  Each searcher is initially considered {\em mobile}. A vertex occupied by a searcher is thereafter {\em protected}, while vertices which have not been visited by searchers are {\em unprotected}. 
 At each stage of the game, searchers on a vertex $v$ move to unprotected neighbours of $v$ if the number of mobile searchers on $v$ is at least the number of its unprotected neighbours.  If the number of mobile searchers on $v$ exceeds the number of unprotected neighbours, the excess searchers may move to any of the unprotected neighbours.   
 Any searchers who have moved $k$ times become {\em immobile} and cannot move again. This process repeats until either all vertices are protected or all searchers are immobile.  Note that $k$ restricts the number of times a searcher moves and not the number of stages it takes for this to happen.  That is, while a searcher becomes immobile after having moved $k$ times, this need not occur in the $k^{\mathrm{th}}$ stage of the game.  As well, two searchers may become immobile at different stages.
 
 The searchers win if at the end of the game, all the vertices are protected; in this case we call the initial layout {\em successful}.  The {\em $k$-move deduction number} of the graph $G$, denoted $d_k(G)$, is the minimum number of searchers for which there exists a successful layout on $G$. If $k=1$, then we omit it from the notation and simply write $d(G)$.

The {\em $k$-move capture time} of a graph $G$ is the minimum number of stages that it takes for all vertices to become protected over all successful layouts with $d_k(G)$ searchers.  If the value of $k$ is clear, we will simply refer to {\em capture time}.  Note that if the $k$-move capture time is equal to $1$, then $d_k(G)=d(G)$.  For this reason, we are primarily interested in cases where the capture time is greater than $1$.

As an illustration of the deduction process, consider the graph $G$ illustrated below.
\begin{center}
\begin{tikzpicture}[x=1cm,y=1cm,scale=1]
\draw (0,1) -- (1,0); \draw (0,-1) -- (1,0);
\draw (1,0) -- (3,0);
\draw (3,0) -- (4,1); \draw (3,0) -- (4,-1);

\draw[fill=black](0,1) circle (3pt) node[left, inner sep=5pt]{$t$};
\draw[fill=black](0,-1) circle (3pt) node[left, inner sep=5pt]{$v$};
\draw[fill=black](1,0) circle (3pt) node[left, inner sep=5pt]{$u$};
\draw[fill=black](2,0) circle (3pt) node[below, inner sep=5pt]{$w$};
\draw[fill=black](3,0) circle (3pt) node[right, inner sep=5pt]{$x$};
\draw[fill=black](4,1) circle (3pt) node[right, inner sep=5pt]{$y$};
\draw[fill=black](4,-1) circle (3pt) node[right, inner sep=5pt]{$z$};
\end{tikzpicture}
\end{center}
Using~\cite[Theorem 5.2]{BDF2024}, it can be seen that $d(G)=5$; a successful layout in $1$-move deduction is given by placing searchers on vertices $t$, $v$, $w$, $y$ and $z$.  
By contrast, if searchers may move twice, three searchers can protect all vertices.  Placing searchers on vertices $t$, $y$ and $w$, in the first stage searchers move from $t$ and $y$ to $u$ and $x$, respectively, and in the second stage the same searchers move to $v$ and $z$. It is easy to see that no successful layout is possible with two searchers, so $d_2(G)=3$.  Similarly we obtain $d_3(G)=3$.  However, moving to $4$-move deduction there is a successful layout with two searchers so that $d_4(G)=2$.  To see this, place searchers on $t$ and $v$.  In the first stage, both searchers move to $u$.  In the second stage, one of these searchers must move to $w$; the second may also do so.  Subsequently, both may move to $x$, and in the final stage one searcher moves to $y$ and the other to $z$, clearing the graph.

The paper~\cite{BDF2024} introduces the $1$-move deduction game, gives various bounds on $d(G)$ and considers the value of the $1$-move deduction number on several classes of graphs including Cartesian products and trees. 
The paper~\cite{BDOWXY} explores connections between deduction and constrained versions of other graph processes, namely zero forcing and fast-mixed search, and gives a structural characterization of graphs with given deduction number and additional lower bounds on $d(G)$.  Nevertheless, determining the value of $d(G)$ is NP-complete~\cite{BDOWXY}.  

The following theorem summarizes the bounds on the $1$-move deduction number from~\cite{BDF2024, BDOWXY}.
\begin{theorem}[\cite{BDF2024,BDOWXY}] \label{1-move bounds}
In any connected graph $G$ of order $n \geq 3$ which contains $\ell$ vertices of degree $1$, 
\[
\max\left\{\left\lceil \frac{n}{2} \right\rceil, \ell, \delta(G), \alpha(G), \omega(G)-1\right\}
\leq d(G) \leq n-1,
\]
where $\alpha(G)$ is the independence number of $G$ and $\omega(G)$ is the clique number of $G$.
\end{theorem}

In this paper, we consider deduction with $k \geq 2$ moves.  In Section 2, we verify that given any successful layout, with the seachers and evaders moving in turn as described in the previous section, the searchers will always capture the evader.   In Section 3, we find the $k$-move deduction number for common families of graphs such as paths, cycles, complete graphs and complete bipartite graphs.   Finally, in Sections 3 and 4 we consider $k$-move deduction for the Cartesian and Strong Products of graphs.

\section{$k$-move deduction as a search model}

Whether we consider the $k$-move deduction model as a cops and robber game or a search problem, we need to demonstrate that every protected vertex can never be safely occupied by a robber or evader.  For $k$-move deduction, it suffices to show that all vertices are eventually protected and if a robber/evader ever occupies a protected vertex, then that vertex is also occupied by a cop/searcher. 

Suppose a set of searchers have a successful layout in a graph $G$ and the capture time is at least two.  Consider any  stage of the game except for the final  stage.  The vertices of $G$ can be partitioned into two sets $V_p$ and $V_u$, where $V_p$ is the set of protected vertices and $V_u$ is the set of unprotected vertices.    Let $\partial (V_p)$ denote the set of vertices in $V_p$ that have a neighbour in $V_u$.  We refer to $\partial (V_p)$ as the {\it  boundary} at that  stage.

\begin{lemma}\label{lem:boundary}
Suppose a set of searchers have a successful layout in $G$ and the capture time is at least two.  At any  stage of the game, except for the final  stage, every vertex in the boundary is occupied by at least one searcher.
\end{lemma}

\begin{proof}
Since stage $t$ is not the final stage of the game, $V_p$ and $V_u$ are both non-empty.  Hence,  $\{V_p,V_u\}$ is a partition of $V(G)$.   

If $t=1$, then $V_p$ represents the set of vertices in the initial layout.  Hence, every vertex of $V_p$ is occupied by a searcher.

Assume $t \ge 2$ and there is some vertex $v \in \partial (V_p)$ that is not occupied by a searcher.  It follows that $v$ was occupied by a searcher at some previous  stage.  Suppose $v$ was occupied by a searcher at stage $t'$, where $1 \le t' <t$ but that searcher moved off of $v$ at stage $t'+1$.  However, the set of searchers on $v$ would only move from $v$ if they could move onto every unprotected neighbour at that stage.  Hence, at stage $t'+1$ every vertex in $N[v]$ would be protected.  Hence, $N[v] \subseteq V_p$ and $v \not \in \partial(V_p)$, which is a contradiction.  
\end{proof}

\begin{corollary}
Given an initial successful layout in $G$, any vertex, once protected, cannot be occupied by an evader without the evader occupying the same vertex as some searcher. 
\end{corollary}

\begin{proof}
Suppose not.  Suppose stage $t$ is the first time an evader moves onto a protected vertex $v$. It follows that $t \ge 2$, since only those vertices occupied by searchers are protected at stage 1.  Therefore, the evader occupied an unprotected $u$ vertex at stage $t-1$.  If $v$ was protected  at stage $t-1$, it follows that it was a boundary vertex, and therefore occupied by a searcher at stage $t-1$. Therefore, at stage $t$, $v$ is either occupied by a searcher or searchers moved from $v$ to all unprotected vertices in $N(v)$.  Therefore, a searcher and evader both occupy vertex $u$ at some stage of the game.  \end{proof}

\section{Initial Results}

We begin this section by considering how the bounds of Theorem~\ref{1-move bounds} generalize to the $k$-move case. 

\begin{lemma}\label{lem:gen_bound_1}
If $G$ has order $n$, then $d_k(G) \geq \left\lceil \frac{n}{k+1}\right\rceil$
and  $d_k(G) \ge \delta (G)$.  Moreover, if $G$ is connected, then $d_k(G) \leq n-1$.
\end{lemma}

\begin{proof}
To see that $d_k(G) \geq \left\lceil \frac{n}{k+1}\right\rceil$, note that each searcher can protect at most $k+1$ vertices: the vertex where they start, as well as up to $k$ that they move to.

A searcher can move from a vertex $v$ in the first stage only if the number of unoccupied vertices in $N(v)$ is less than or equal to the number of searchers on $v$ in the initial layout.  Therefore, $|N(v)| \le d_k(G)$ and $d_k(G) \ge \delta (G)$.

For any connected graph $G$, placing a searcher on all but one vertex gives a successful layout in $k$-move deduction for any $k \geq 1$, showing $d_k(G) \leq n-1$.
\end{proof}

Additionally, a similar argument to~\cite[Theorem~4.5]{BDF2024} shows the following.
\begin{lemma} \label{lemma:clique}
If $G$ is a graph of order $n$, then $d_k(G) \geq \omega(G)-1$.
\end{lemma}

The bounds of Lemmas~\ref{lem:gen_bound_1} and~\ref{lemma:clique} are tight for every value of $k$.  To see this, note that the path $P_n$ of order $n$ has $d_k(P_n)=\left\lceil \frac{n}{k+1}\right\rceil$ and the complete graph $K_n$ has $d_k(K_n) = \delta(K_n) = \omega(K_n)-1=n-1$.

We note, however, that not all properties of the $1$-move game translate well into the $k$-move variant, including several of the other lower bounds from Theorem~\ref{1-move bounds}.  For example, the path $P_{k+1}$ of length $k$ shows that $d_k(G)$ is not bounded below by $\alpha(G)$ or the number $\ell$ of vertices of degree $1$ in $G$ in general, as in this instance $d_k(P_{k+1}) = 1$ but $\alpha(P_{k+1})=\lceil\frac{k+1}{2}\rceil$ and $\ell=2$.

As another example to illustrate the differences between $1$-move and $k$-move deduction, it is shown in~\cite{BDOWXY} that if $H$ is an induced subgraph of $G$, then $d(H) \leq d(G)$.  The following result shows that it is possible to have an induced subgraph $H$ for which $d_k(H)>d_k(G)$ if $k >1$.  

\begin{lemma}
  For any $k \ge 8$, there exists a graph $G$ with induced subgraphs $H_1$ and $H_2$ such that $d_k(H_1) < d_k(G)$ and $d_k(H_2) > d_k(G)$.
\end{lemma}

\begin{proof}
For any $k \ge 8$, let  $m=\lfloor \frac{k}{2} \rfloor$ and $T_m$ be a tree constructed as follows: begin with a vertex $v$ and add $m$ paths, each of length $m$, all starting at $v$.  Then $T_m$ has the form of a spider with legs of length $m$.

We know that at least $\left \lceil \frac m2 \right \rceil$ searchers are required to protect $T_m$ since the vertices protected by any one searcher induce a path in $T_m$.  Therefore, one searcher can protect at most two of the branches radiating from $v$.  Furthermore, if you select $\left \lceil \frac m2 \right \rceil$ leaves and place a searcher on each, those searchers can protect $T_m$ in $2m$ moves.  Since $k \ge 2m$, it follows that $d_k(T_m) = \left \lceil \frac m2 \right \rceil$.  

Now consider the induced subgraph $T_m$ obtained by deleting $0, 1, \ldots , m-1$ vertices, respectively, from the ends of each branch.  The resulting graph $T'$ has branches of length $1, 2, \ldots, m$ radiating from $v$. We claim that this graph requires at least $m-1$ searchers.

Suppose $n$ searchers are sufficient, where $n\le m-2$.  Then there are at least two searchers that each clear at least two branches.   It follows that these two searchers $A$ and $B$, each start on a leaf.  Each searcher moves closer to $v$ in each stage.  Since all the branches are different lengths, one searchers arrives at $v$ before the other.  Without loss of generality, suppose $A$ arrives first.  It follows, from the rules of the game, that $B$ cannot subsequently move onto $v$.  Therefore, it is impossible for $B$ to clear a second branch.   Hence, at most one searcher can clear more that one branch and $d_k(T_m') \ge m-1$ and $k \ge m$.  It is straightforward to show that $d_k(T_m') = m-1$ by placing a single searcher on all leaves of $T_m'$ except for the leaf adjacent to $v$.  

It follows that, for any  $m \ge 4$ and $k\ge 2m$, $d_k(T_m) < d_k(T_m')$.  

We can also consider the induced subgraph of $T_m$ obtained by removing two branches of $T_m$. Call this $T_m''$.  It follows that $d_k(T_m'') =  d_k(T_m)-1$ for any $k \ge 2m$.  
\end{proof}

We conclude this section by determining the $k$-move deduction number for several basic classes of graphs, namely paths, cycles, complete graphs and complete bipartite graphs. The case $k=1$ for these graphs was considered in~\cite{BDF2024}.  Recall that we denote by $P_n$ the path on $n$ vertices.

\begin{lemma}\label{lem:path_cycle_complete}
If $k \ge 2$, then
\begin{enumerate}
    \item  $d_k(P_n) = \left\lceil \frac{n}{k+1}\right\rceil$ for any $n \ge 2$;
    \item  $d_k(C_n) =\max\{2,  \left\lceil \frac{n}{k+1}\right\rceil \}$ for any $n \ge 3$;
    \item $d_k(K_n) = n-1$ for any $n \ge 2$.
\end{enumerate}
\end{lemma}

\begin{proof}
We prove each case separately. 

\begin{enumerate}
    \item 
Let $P_n = v_0v_1\cdots v_{n-1}$.  Place a single searcher on each vertex in the set $\{v_{j(k+1)} : j = 0,  \ldots, \left \lceil \frac{n}{k+1} \right \rceil -1\}$. This is a successful initial layout and is equivalent to clearing $ \left\lfloor \frac{n}{k+1}\right\rfloor$ paths of length $k$ and potentially one path of length at most $k-1$.

\item By Lemma \ref{lem:gen_bound_1}, it follows that $d_k(C_n) \ge 2$ for any $k \ge 2$ and $n \ge 3$.

Let $C_n = v_0v_1\cdots v_{n-1}v_0$.  If $n \le 2(k+1)$, place a searcher on each of $v_0$ and $v_{n-1}$.  The searcher that starts on $v_0$ moves along the path $P = v_0v_1 \ldots v_k$ and the searcher that starts on $v_{n-1}$ moves along the path   $Q = v_{n-1}v_{n-2} \ldots v_{n-k}$. Note that if $n < 2(k+1)$ the searchers only move until all vertices are cleared and may not traverse the entire paths $P$ and $Q$.    

If $n > 2(k+1)$, place a single searcher on each vertex in the set $\{v_{j(k+1)} : j = 0,  \ldots, \left \lceil \frac{n}{k+1} \right \rceil -2\} \cup v_{n-1}$. The searcher on $v_0$ will traverse the path $P = v_0v_1\cdots v_{k}$. Thereafter the searcher on $v_{k+1}$ will be able to move. This repeats with all searchers until the graph is cleared.

\item  By Lemma~\ref{lem:gen_bound_1}, since $\delta(G) = n-1$ we have that $d_k(K_n)\geq n-1$. Place $n-1$ searchers on different vertices. Then on the first turn, each searcher can move to the only uncovered vertex. Thus, $d_k(K_n) =n-1$. 
\end{enumerate}
\end{proof}

\begin{theorem}
If $k\geq 2$, then $d_k(K_{m,n}) = m+  \left\lceil \frac{n}{2}\right\rceil -1$  for any $n \ge m \ge 1$.
    
\end{theorem}
\begin{proof}
We note that  $d_k(K_{1,1}) = 1$, and it therefore follows that $d_k(K_{m,n}) = m-1+  \left\lceil \frac{n}{2}\right\rceil $  when $m=n=1$.  We will therefore assume that $n \ge 2$.  

    Let $K_{m,n}$ have bi-partition $(X, Y)$, with $|X|=m$, $|Y|=n$, and $1 \le m \le n$. 
    
      Let $Y'$ be a subset of $Y$ such that $|Y'| = \left \lceil \frac n2 \right \rceil$.  Now select a single vertex $x \in X$.  Place a single searcher on each vertex of $Y' \cup (X - x)$.  This is the initial layout, using $m - 1 + \left \lceil \frac n2 \right \rceil$ searchers.  
    
    In round 1, the searchers on $Y'$ all move to $x$. At the beginning of round 2, there are $\left \lceil \frac n2 \right \rceil$ searchers on $x$ and $\left \lfloor \frac n2 \right \rfloor$ uncleared vertices in $Y$.  Therefore, the searchers move from $x$ to the uncleared vertices, $Y-Y'$, so that there is at least one searcher on each vertex of $Y-Y'$.  At the end of round 2, all vertices are cleared. Hence, $d_k(K_{m,n}) \le m-1+  \left\lceil \frac{n}{2}\right\rceil$

Next, we show any optimal successful layout requires at least  $m-1+  \left\lceil \frac{n}{2}\right\rceil$ searchers.

Consider an optimal initial layout of searchers on a set of vertices.  Call the set $L_0$.  Let $S = L_0 \cap X \neq \emptyset$,  $T = L_0 \cap Y$,  $|S|=s$, and $|T| = t$.  Furthermore, let $S' \subseteq S$ and $T' \subseteq T$ such that exactly those searchers on $S' \cup T'$ move in round 1.  Let $s'=|S'|$ and $t'=|T|$.

\noindent Case 1:  Suppose $S'$ and $T'$ are both non-empty.  This means each vertex in $S'$ must have $|Y-T| = n-t$ searchers placed on it in round 0.  Similarly, each vertex in $T'$ has $m-s$ searchers placed on it in round 0.  Since each vertex in $S \cup T$ has at least one searcher, it follows that at least $(n-t)+ (m-s) + s-1+t-1 = n+m-2$ searchers are used.  Since we have an optimal initial layout, it follows that $d_k(K_{m,n}) \ge n+m-2$. Since $n \ge 2$, it follows $\lfloor \frac n2 \rfloor \ge 1$.  Hence, $n+m-2  = \lceil \frac n2 \rceil + \lfloor \frac n2 \rfloor + m -2 \ge  m + \lceil \frac n2 \rceil - 1$ and it follows that $d_k(K_{m,n}) \ge  m + \lceil \frac{n}{2} \rceil -1$.

\noindent Case 2: Suppose $S'$ is non-empty and $T'$ is empty.  It follows that $n>t$, otherwise no searcher in $S$ can move onto $Y$. 
  
In order to move in round 1, each vertex in $S'$ must have at least $n-t$ searchers on it.  Therefore, at the end of round 1, at least $s'$ searchers have moved onto each vertex of $Y-T$.  Furthermore, at the end of round 1, there is at least one searcher on each vertex of $T$, no searchers on $S'$,  and at least one searcher on each vertex of $S-S'$ .  From just this analysis, we know that at least $s'(n-t) + t + s-s'$ searchers are required. 

Case 2a:  $m-s \le s'$.  Recall that at the beginning of round 2, there were at least $s'$ searchers on each vertex of $Y-T$. Since $m-s \le s'$, it follows that all the searchers in $Y-T$ can move onto $X-S$ and clear $X-S$.  Therefore, the graph is cleared in two rounds using at least $s'(n-t)+ s-s'+ t$ searchers and $d_k(G) \ge s'(n-t-1) + s + t$.

Since $s' \ge 1$ and $n-t \ge 1$, $s'(n-t-1 )+ s+t  \ge n+s -1$.  Furthermore, $m -s \le s'$ implies $m \le s+s' \le 2s$ and $s \ge \left \lceil \frac m2 \right \rceil$.  Therefore,  $d_k(K_{m,n}) \ge n + \left \lceil \frac m2 \right \rceil -1  = \lceil \frac m2 \rceil +\lfloor \frac n2 \rfloor + \lceil \frac n2 \rceil  -1 \ge m + \left \lceil \frac n2 \right \rceil -1$.

Case 2b: $m-s > s'$.  It follows that, at the end of round 1, at least one vertex in $Y-T$ has $m-s$ searchers on it.  Therefore, we need $m-s-s'$ searchers in addition to the $s'(n-t-1)+ s+ t$ searchers previously justified.  Therefore,  $d_k(K_{m,n}) \ge m+ s'(n-t-2) + t$.

If  $n -t \ge 2$,  then $m+ t+ s'(n-t-2) \ge m+n -2$, since $s' \ge 1$.  As in Case 1,   $n+m-2 \ge  m + \lceil \frac n2 \rceil - 1$ since  $n \ge 2$.   It follows that $d_k(K_{m,n}) \ge  m + \lceil \frac{n}{2} \rceil -1$.   

If $n -t = 1$, then every searcher on $S$ can move onto the one uncleared vertex in $Y$ in round 1.  Therefore, $S = S'$ and  $m+ t+ s'(n-t-2) = m+t -s = m+n-1-s$.  Therefore,  $d_k(K_{m,n}) \ge m-s+n-1$.

Since $s=s'$ and $m-s>s'$,  we have $s \le \left \lfloor \frac m2 \right \rfloor \le \left \lfloor \frac n2 \right \rfloor $.  Hence, $m-s + n-1 \ge m +\left \lceil \frac n2 \right \rceil  -1$ and  $d_k(K_{m,n}) \ge m+ \lceil  \frac n2 \rceil  -1$.  

It follows that whenever $T'$ is empty,  $d_k(K_{m,n}) \ge m + \lceil \frac{n}{2} \rceil -1$.  

Case 3:  Suppose $S'$ is empty and $T'$ is non-empty.  

Case 3a: $n-t \le t'$.  Using arguments similar to those in Case 2a, we can show that when $n-t \le t'$, that $d_k (K_{m,n}) \ge t'(m-s-1) + t+s \ge m+t -1 \ge m + \lceil \frac n2 \rceil -1$.  

Case 3b: $n-t>t'$. Using arguments similar to those in Case 2b, we can show that $d_k(K_{m,n} )\ge n+m-2 \ge m + \lceil \frac n2 \rceil -1$ when $n-t \ge 2$.  Finally, when $n-t = 1$, we can show that $d_k(K_{m,n}) \ge n-t+m-1$,  $t=t'$ and $t \le \lfloor \frac n2 \rfloor $.  It follows that $d_k(K_{m,n}) \ge m+ \lceil \frac n2 \rceil -1 $.  

\medskip 

In all cases we have $d_k(K_{m,n}) \ge m+ \lceil \frac n2 \rceil -1 $ and it therefore follows that $d_k(K_{m,n}) = m+ \lceil \frac n2 \rceil -1 $.
\end{proof}

\section{Cartesian Products}

For any two graphs $G$ and $H$, the {\em Cartesian product} of those graphs, denoted $G \Box H$, is defined as follows: $V(G \Box H) = \{(u,v) : u \in V(G) \mbox{ and } v \in V(H)\}$ and $E(G \Box H) = \{(x,y)(u,v) : x = u \mbox{ and } yv \in E(H), \mbox{ or } xu \in E(G) \mbox{ and } y=v\}$.   The value of $d(G \Box H)$ was considered in~\cite{BDF2024}, where it was shown that $d(G \Box H) \leq \min\{|V(G)| \cdot d(H), |V(H)| \cdot d(G)\}$.  We extend this upper bound to $k$-move deduction and also establish lower bounds on $d_k(G \Box H)$ for general graphs $G$ and $H$ in Section~\ref{GenCartesian}.  %first establish upper and lower bounds on $d_k(G \Box H)$ for all $k \ge 1$ for general graphs $G$ and $H$ and 
We then in Section~\ref{CartesianPath} look at the specific case where $G$ and $H$ are paths.  

\subsection{Upper and Lower Bounds on $d_k(G \Box H)$} \label{GenCartesian}

We will say that a subgraph is {\it totally protected} if all vertices of the subgraph are protected and  {\it unprotected} if no vertex is protected.  Otherwise, we will say that the subgraph is {\it partially protected}. We say that a subgraph is {\it protected} if it is partially or totally protected.

\begin{theorem}\label{lower_bound_Cartesian}
    Suppose $2 \le m \le n$, and let $G$ and $H$ be connected graphs of order $m$ and $n$, respectively.  For any $k \ge 1$, $$d_k(G \Box H) \ge \min \{ n+  \delta (G) -1, m+  \delta (H) -1\}.$$
\end{theorem}

\begin{proof}
   Consider a successful initial layout using $d_k(G \Box H)$ searchers and the resulting moves at each stage.  
   
   If the capture time of $G \Box H$ is $1$, then $d_k(G \Box H) = d(G \Box H)$ and by Theorem~\ref{1-move bounds}, $d_k(G \Box H) \geq \frac{mn}{2}$.  If $m \geq 4$, then $\frac{mn}{2} \geq 2n \geq m+n > m+ \delta(H)-1$.  If $m=3$, then $\delta(G) \leq 2$, and $\frac{mn}{2} = \frac{3n}{2} > n+1 \geq n+ \delta(G)-1$.
   If $m=2$, then $\delta(G)=1$ and $\frac{mn}{2}=n=n+\delta(G)-1$.
   We thus assume that the capture time is at least $2$.

Let $V(G) = \{x_1, x_2, \ldots, x_m\}$ and $V(H) = \{y_1, \ldots , y_n\}$.  For each $j = 1, \ldots n$, let $G_j$ be the subgraph of $G \Box H$ induced on $\{(x, y_j): x\in V(G)\}$.  Similarly, for each $i = 1, \ldots, m$, let $H_i$ be the subgraph induced on $\{(x_i, y): y \in V(H)\}$.  Let ${\cal G} = \{G_j: j=1, \ldots, n\}$ and ${\cal H} = \{H_i: i=1, \ldots, m\}$.

{\it Case 1: In the initial layout, there is at least one searcher on each copy of $G$ in ${\cal G}$, or at least one searcher on each copy of $H$ in ${\cal H}$.}   

Suppose there is at least one searcher in each copy of $G$ in the initial layout.  Also assume, without loss of generality, that the vertex $v=(x_1,y_1)$  is occupied by a searcher in the initial layout, but not occupied by a searcher at stage 1.  It follows that, in the initial layout, there is one searcher on $v$ for each unprotected vertex in $N(v)$. Therefore, the number of searchers appearing in $N[v]$ in the initial layout is at least $\mbox{deg }(v) = \mbox{deg}_G (x_1) + \mbox{deg}_H (y_1)$.  
Since $N[v] \cap V(G_j) = \emptyset$ whenever $y_j \not \in N_H[y_1]$,  there are $n - \mbox{deg}_H (y_1) -1$ copies of $G \in {\cal G}$ that do not intersect $N[v]$.  It follows that there are at least $\mbox{deg }(v) + n - \mbox{deg}_H (y_1) -1$ searchers.  Hence, $d_k (G \Box H) \ge n + \mbox{deg}_G (x_1) - 1 \ge n+ \delta (G) -1 $.

If, in the initial layout, there is at least one searcher on each copy of $H$ in ${\cal H}$, it can be similarly shown that $d_k(G \Box H) \ge m+ \delta (H) -1$.  Hence, in Case 1, we have $d_k(G \Box H) \ge \min \{ n+  \delta (G) -1, m+  \delta (H) -1\}$.

{\it Case 2: In the initial layout, there is some copy of $G$ in ${\cal G}$ and some copy of $H$ in ${\cal H}$ such that neither is occupied by a searcher.}

It follows that there are copies of $G$ and $H$ that are initally unprotected.  Suppose stage $t$ is the last stage in which there are subgraphs $G_u \in {\cal G}$ and  $H_u \in {\cal H}$ that are unprotected.   Without loss of generality, assume that at stage $t+1$ every $H \in {\cal H}$ is at least partially protected.  

Let ${\cal G}_p$ be the subset of ${\cal G}$ consisting of copies of $G$ that have at least one protected vertex at stage $t$.  Define ${\cal H}_p$ similarly. It follows that ${\cal G}_p$, ${\cal H}_p$, ${\cal G} - {\cal G}_p$, and ${\cal H} - {\cal H}_p$ are all non-empty.

\begin{claim}
    Every $G \in {\cal G}_p$ and every $H \in {\cal H}_p$ is partially protected.
\end{claim}

\begin{claimproof}
    Without loss of generality, suppose there is some $H_p \in {\cal H}_p$ such that $H_p$ is totally protected at stage $t$. It follows that every $G \in {\cal G}$ is partially protected, since each $G \in {\cal G}$ contains exactly one vertex from each of $H_p$ and $H_u$. %It follows that, 
    Thus ${\cal G}_p = {\cal G}$ and ${\cal G}_u = \emptyset$, which is a contradiction.   
\end{claimproof}

It follows that every $G \in {\cal G}_p$ and every $H \in {\cal H}_p$ has a boundary vertex at stage $t$, and is therefore occupied by at least one searcher.  

Without loss of generality, suppose a searcher moves from a vertex $v$ in $H_1 \in {\cal H}_p$ to a vertex $w$ in $H_2 \in {\cal H} - {\cal H}_p$ at stage $t+1$.   It follows that, at stage $t$, $v$ is occupied by at least one searcher for each of its unprotected neighbours in $V(H_1)$, as well as a searcher that moves to $w$.  Furthermore, at stage $t$, each protected vertex in $H_1$ is a boundary vertex since it has a neighbour in $H_2$ which is unprotected.  Hence, every protected vertex in $H_1$ is occupied by a searcher at stage $t$.  Therefore, $N[v] \cap V(H_1)$ is occupied by at least $\mbox{deg}_H y_1 + 1 \ge \delta (H) +1$ searchers.   

Suppose $|{\cal H} - {\cal H}_p| \ge 2$.  Without loss of generality, suppose $H_3 \in {\cal H} - {\cal H}_p$.  If $x_1$ is adjacent to $x_3$ in $G$, it follows that there must be an additional searcher on $v$ to move from $v$ to its unprotected neighbour in $H_3$.    If $x_1$ is not adjacent to $x_3$ in $G$, then searchers move from some $H' \in {\cal H}_p$ to $H_3$ at stage $t+1$, which means there must be at least  $\delta (H) +1$ searchers on $H'$ at stage $t$.

We therefore have the following situation at stage $t$:
\begin{enumerate}
    \item  each copy of $H$ in ${\cal H}_p$ is occupied by at least one searcher,
    \item for the first copy of $H$ in ${\cal H} - {\cal H}_p$, there must be least $\delta (H)$ additional searchers,
    \item for each subsequent  copy of $H$ in ${\cal H} - {\cal H}_p$, there must be at least one additional searcher.
\end{enumerate}  

It follows that there are at least $|{\cal H}_p| + \delta(H) + (m-|{\cal H}_p|-1) = m+\delta (H) -1$ searchers on $G \Box H$.  Hence, $d_k(G \Box H) \ge  m+  \delta (H) -1\ge \min\{ n+  \delta (G) -1, m+  \delta (H) -1\}$.
\end{proof}

\begin{theorem}\label{upper_bound_Cartesian}
Suppose $G$ and $H$ are connected graphs of order $m$ and $n$, respectively.  For any $k \ge 1$, $$d_k(G \Box H) \le \min\{md_k(H), n d_k(G)\}.$$
\end{theorem}

\begin{proof}
    Consider any  initial successful layout in $G$ using $d_k(G)$ searchers.  Now, for each copy of $G$ in $G \Box H$ place $d_k(G)$ searchers on the corresponding initial layout, using $nd_k(G)$ searchers in total.  Now, in each round, the searchers all move within the copy of $G$ in which they were initially placed, moving exactly as they would in a game played on $G$ alone.   
\end{proof}

\bigskip

\subsection{Cartesian Products of Paths} \label{CartesianPath}

In this section, we consider the $k$-move deduction game on the Cartesian product of paths.  Recall that $P_n$ denotes a path of order $n$.  We give bounds on the value of $d_k(P_m \Box P_n)$ with $m \leq n$, and determine its exact value in several cases. We begin with two results that follow from  Theorems \ref{lower_bound_Cartesian} and  \ref{upper_bound_Cartesian}, respectively.  

\begin{corollary}\label{Cor:unlimited_moves}
    Suppose $2 \le m \le n \le k+1$.  Then $d_k(P_m \Box P_n) = m$.
\end{corollary}

\begin{proof}
    By Theorem \ref{lower_bound_Cartesian}, $d_k(P_m \Box P_n) \ge m$.  By Theorem \ref{upper_bound_Cartesian}, $d_k(P_m \Box P_n) \le m\cdot d_k(P_n)$.  Since $n\le k+1$, it follows that $d_k(P_n) = 1$ and $d_k(P_m \Box P_n) \le m$.  Hence, $d_k(P_m \Box P_n) = m$. 
\end{proof}

\bigskip

\begin{corollary}\label{cor:0mod}
 Suppose $m,n \ge 2$.  If $m \equiv 0 \pmod{k+1}$ or $n \equiv 0 \pmod{k+1}$, then $d_k (P_m \Box P_n) =    \frac{nm}{k+1}$.
\end{corollary}

\begin{proof}
Without loss of generality, assume that $n \equiv 0 \pmod{k+1}$.  It follows from Theorem \ref{upper_bound_Cartesian} and Lemma \ref{lem:path_cycle_complete} that $d_k (P_m \Box P_n) \le m \cdot d_k(P_n)= m \cdot \frac{n}{k+1} $. 
   Using the lower bound from Lemma \ref{lem:gen_bound_1}, we also have $d_k (P_m \Box P_n) \ge \frac{nm}{k+1}$.  The result follows.  
\end{proof}

\bigskip

Corollaries  \ref{Cor:unlimited_moves} and \ref{cor:0mod} give us the value of $d_k(P_m \Box P_n)$ exactly for all $m$ and $n$ such that $2 \le m \le n \le k+1$ or $m = k+1 <n$.  
The next subset of Cartesian grids that we are interested in examining are the the grids $P_m \Box P_n$  when $2 \le m \le k$ and $n>k+1$.  For some of these grids, we can improve on the lower bounds given in Lemma \ref{lem:gen_bound_1}.  This improvement is given in Corollary \ref{cor:improve_lower}, but first we introduce some notation.

For subsequent results in this section, we label the vertices of $P_m \Box P_n$ (the $m \times n$ Cartesian grid) using the set $\{(i,j) \mid 0 \le i \le m-1, \; 0 \le j \le n-1\}$, where $(i,j)$ and $(i',j')$ are adjacent if $|i-i'|\le 1$ and $j=j'$, or $|j-j'| \le 1$ and $i=i'$.

For each $j=0, \ldots , n-1$, let $R_j =\{(i,j): i=0, \ldots, m-1\}$ and refer to it as \textit{row} $j$ of the grid.  Similarly, let $C_i = \{(i,j): j=1, \ldots, n-1\}$ for each $i=1, \ldots, m-1$ and refer to it as  \textit{column} $i$ of the grid.

\bigskip

\begin{lemma}\label{lem:another_lower}
    Suppose $2 \le r \le k$ and $m,n \ge r$.  For any subgraph $H \cong P_r \Box P_r$  in $P_m \Box P_n$, a set of $d_k(P_m \Box P_n)$ searchers can protect $P_m \Box P_n$ only if at least $r$ of those searchers each move onto at least one vertex of $H$.
\end{lemma}

\begin{proof}
    Let $G = P_m \Box P_n$.  Suppose there is a successful initial layout in $G$ such that for some $r$, where $2 \le r \le k$, there is a subgraph $H \cong P_r \Box P_r$ occupied by fewer than $r$ searchers.   It follows that there are at least one row of $H$ and one column of $H$ that are not occupied by any searcher.

    Suppose that at the end of stage $t$ there is at least one row of $H$ and at least one column in $H$ that are unprotected, but at the end of stage $t+1$ all rows or all columns are protected.    Without loss of generality, assume all columns are protected at the end of stage $t+1$.

    Suppose that at the end of stage $t$ there are $a>0$ unprotected columns and $r-a$ protected columns.  Therefore, at stage $t+1$, at least $a$ searchers move onto unprotected columns.  
    
    Of the $a$ columns that are unprotected at stage $t$, suppose $b$ of those are protected at stage $t+1$ as a result of a searcher moving within $H$ and the remaining $a-b$ columns are protected as a result of a searcher moving onto $H$ from a vertex outside of $H$.
    
    As in the proof of Theorem \ref{lower_bound_Cartesian}, no column is fully protected and every partially protected column is occupied by a searcher.  Furthermore, at stage $t+1$, if a searcher moves from a partially protected column to an unprotected column that means there were at least two searchers on the partially protected column at stage $t$.   Hence, there are at least $b + r-a$ searchers in $H$ at stage $t$, and an additional $a-b$ searchers that move onto $H$ from outside of $H$ at stage $t+1$.  Therefore, by stage $t+1$, at least $r$ different searchers have moved onto some vertex of $H$.
\end{proof}

\begin{corollary}\label{cor:improve_lower}
Suppose $1 \le r,s \le k$.  If $s \ge r$, then $d_k (P_r \Box P_{k+1+s}) \ge r + \min\{s-r+1, r\}$.
%Furthermore, if $s \ge 2r-1$, then $d_k (P_r \Box P_{k+1+s}) = 2r$. 
\end{corollary}
\begin{proof}
    Consider the two subgraphs of $P_r \Box P_{k+1+s}$ induced by the sets $\{(i,j): 0 \le i \le r-1, 0 \le j \le r-1\}$ and $\{(i,j): 0 \le i \le r-1, k+r \le j \le k+s\}$, respectively.  Refer to these subgraphs as $H_1$ and $H_2$, respectively, and note that $H_1 \cong P_r \Box P_r$ and $H_2 \cong P_r \Box P_{s-r+1}$.  It follows from Lemma \ref{lem:another_lower}, at least $r$ searchers move onto vertices of $H_1$ and at least $\min\{s-r+1, r\}$ searchers move onto vertices of $H_2$.  

    Since each vertex in $H_1$ is at least distance $k+1$ from every vertex in $H_2$, then no searcher can protect vertices in both $H_1$ and $H_2$.  Therefore, at least $r + \min\{s-r+1, r\}$ searchers are required to protect $P_r \Box P_{k+1+s}$ and
    $d_k (P_r \Box P_{k+1+s}) \ge r + \min\{s-r+1, r\}$. 
\end{proof}

%\bigskip

The remaining results in this section provide a number of different successful initial layouts that yield upper bounds on the $k$-move deduction number.

\begin{figure}[ht]
\begin{center} \includegraphics[width = .5\linewidth]{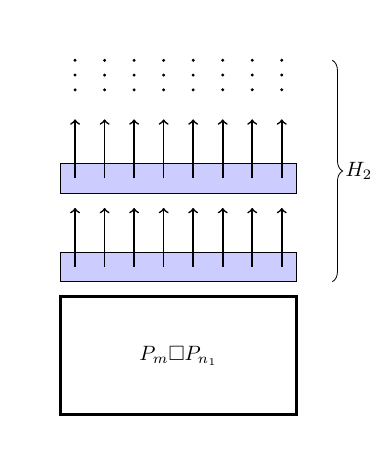}\end{center}
\caption{$P_m \Box P_n$ can be partitioned into two subgrids $P_m \Box P_{n_1}$ and $H_2\cong P_m \Box P_{n_2}$}\label{fig:partition}\end{figure}

\begin{lemma}\label{lemma:partition}
   Suppose $m \ge 1$ and $n > k+1$.  Then $d_k(P_m \Box P_n) \le d_k(P_m \Box P_{n_1}) + d_k(P_m \Box P_{n_2}) $ for positive integers $n_1$ and $n_2$ such that $n_1+n_2=n $  and $n_2 \equiv 0 \bmod (k+1)$ .
\end{lemma}

\begin{proof}
 Let $G = P_m \Box P_n$ where $m \ge 1$ and $n > k+1$.  Partition the vertices of $G$ into two sets: $S_1 = \{(i,j): 0\le i \le m-1, 0 \le j \le n_1 -1\}$ and $S_2 = \{(i,j): 0\le i \le m-1, n_1 \le j \le n-1\}$.   Let $H_1$ and $H_2$ be the subgraphs induced on $S_1$ and $S_2$, respectively.   It follows that $H_1 \cong P_{m} \Box P_{n_1}$ and $H_2 \cong P_{m} \Box P_{n_2}$.  
 
 Since $n_2 \equiv 0 \bmod{(k+1)}$, we know from the proof of Theorem \ref{upper_bound_Cartesian} combined with the results of Lemma~\ref{lem:path_cycle_complete} and Corollary~\ref{cor:0mod}, that there is an initial successful layout of $d_k(P_{m} \Box P_{n_2})$ searchers in $H_2$ such that every vertex in row 0 of $H_2$ is occupied by a searcher and no two searchers occupy the same vertex.  Let $T_2$ be the set of vertices in $G$ corresponding to such an initial successful layout in $H_2$.    In $H_1$, consider any initial successful layout that uses $d_k(P_{m} \Box P_{n_1})$ searchers.  Let $T_1$ be the set of vertices in $G$ corresponding to such an initial successful layout in $H_1$.  

 Consider the initial layout in $G$ where searchers are placed on  $T_1 \cup T_2$ as described above.  Since every edge from a vertex in $S_1$ to a vertex in $S_2$ has its endpoint in $S_2$ occupied by a searcher, no searcher moves from a vertex in $S_1$ to a vertex in $S_2$ during the course of the game. Furthermore, no searcher on a vertex $R_{n_1} \cap T_2$ moves until after the vertex adjacent to it in row $R_{n_1-1}$ of $G$ has been protected.  Therefore, no searcher in $S_2$ will move onto a vertex in $S_1$ during the course of the game.  It follows that the searchers starting in $T_1$ will move exactly as they do in $H_1$ during the course of the game.

 The searchers starting in $T_2$ will move along the same paths in $G$ as they would in $H_2$ (see Figure~\ref{fig:partition}), but not necessarily at the same stages.  However, since no two searchers ever occupy the same vertex, timing is irrelevant.   The result follows.
\end{proof}

\begin{corollary}\label{cor:improve_2}
Suppose $1 \le r,s \le k$.  If $s \ge 2r-1$, then $d_k (P_r \Box P_{k+1+s}) = 2r$.
\end{corollary}

\begin{proof} Suppose $s \ge 2r-1$.
It follows from Lemma \ref{lemma:partition}, that $d_k (P_r \Box P_{k+1+s}) \le d_k(P_r \Box P_s) + d_k(P_r \Box P_{k+1}) = \min\{s,r\} + r = 2r$.  Furthermore, by Lemma \ref{lem:another_lower},  $d_k (P_r \Box P_{k+1+s}) \ge r+ \min\{s-r+1, r\}=2r$.  
\end{proof}
\begin{theorem}\label{Upper_bound_Cartesian2}

   For any $m,n \ge 2$, $$d_k(P_m \Box P_n) \le \frac{mn -r_1r_2}{k+1} + \min\{r_1, r_2\}$$ where $r_1 = m \bmod {(k+1)}$ and $r_2 = n \bmod {(k+1)}$.
\end{theorem}

\begin{proof}
Suppose $2 \le m \le n$, $r_1 = m \bmod {(k+1)}$ and $r_2 = n \bmod {(k+1)}$.    By Corollary \ref{cor:0mod}, the result  holds when $r_1 = 0$ or $r_2 =0$.  Therefore, we will assume that $r_1, r_2 \ge 1$.

Let $q_1 = \frac{m-r_1}{k+1}$ and $q_2 = \frac{n-r_2}{k+1}$.  We now consider cases based on the values of $q_1$ and $q_2$.  In all cases we will show that $d_k(P_m \Box P_n) \le mq_2 + r_2q_1 + \min\{r_1, r_2\}$.

Case 1:  {\it  $q_1 = 0$ or $q_2 = 0$}.

If $q_1 = q_2 = 0$, then by Corollary \ref{Cor:unlimited_moves},  $d_k(P_m \Box P_n) = \min\{m,n\} = \min\{r_1,r_2\}$.  Therefore, $d_k(P_m \Box P_n) =  mq_2+r_2q_1 +\min\{r_1,r_2\}$

If $q_1 = 0$ and $q_2 \ge 1$, then by Lemma \ref{lemma:partition}, $d_k(P_m \Box P_n) \le d_k(P_m \Box P_{r_2}) + d_k(P_m \Box P_{n-r_2}) = \min \{r_1, r_2\} + mq_2 = mq_2+ r_2q_1 + \min\{r_1, r_2\}$.

Similarly, if $q_2 = 0$ and $q_1 \ge 1$, 
$d_k(P_m \Box P_n) \le \min \{r_1, r_2\} + nq_1 = mq_2+ r_2q_1 + \min\{r_1, r_2\}$.

Case 2:  $q_1, q_2 \ge 1$

By Lemma \ref{lemma:partition}, we have $d_k(P_m \Box P_m) \le d_k(P_m \Box P_{r_2}) + d_k(P_m \Box P_{n-r_2})$.  Furthermore, $d_k(P_m \Box P_{r_2}) =  d_k( P_{r_2} \Box P_m) \le d_k( P_{r_2} \Box P_{r_1}) + d_k( P_{r_2} \Box P_{m-r_1})$.  It follows that $d_k(P_m \Box P_n) \le d_k(P_m \Box P_{n-r_2}) + d_k( P_{r_2} \Box P_{m-r_1})+d_k( P_{r_2} \Box P_{r_1}) = mq_2 + r_2q_1 + \min\{r_1, r_2\}$.

Hence, in all cases $d_k(P_m \Box P_n) \le mq_2+ r_2q_1 + \min\{r_1, r_2\}$.   Since $mn= m(q_2(k+1)+r_2) = (k+1)(mq_2) + mr_2 =(k+1)(mq_2)+ (k+1)q_1r_2 + r_1r_2$, it follows that $\frac{mn-r_1r_2}{k+1} = mq_2+r_2q_1$.  Hence, $d_k(P_m \Box P_n) \le \frac{mn-r_1r_2}{k+1} + \min\{r_1, r_2\}$.
\end{proof}

Figure~\ref{fig:Cartesian Product} shows a successful layout with nine searchers on $P_5 \Box P_5$ with $k=2$.  Here $r_1=r_2=2$, so that $\frac{mn-r_1r_2}{k+1}+\min\{r_1r_2\} = \frac{25-4}{3}+2=9$.

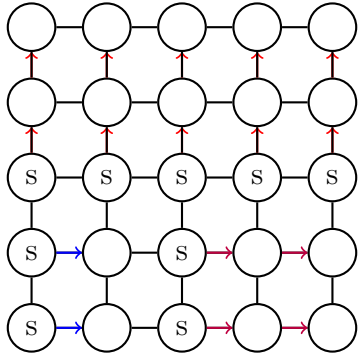
\begin{figure}[ht]
\begin{center}
\begin{tikzpicture}
\begin{scope}[every node/.style={circle,thick,draw}]
    \node (1) at (0,0) {s};
    \node (2) at (1,0) {\phantom{s}};
    \node (3) at (2,0) {s};   
    \node (4) at (3,0) {\phantom{s}};   
    \node (5) at (4,0) {\phantom{s}};
    \node (6) at (0,1) {s};
    \node (7) at (1,1) {\phantom{s}};
    \node (8) at (2,1) {s};
    \node (9) at (3,1) {\phantom{s}};
    \node (10) at (4,1) {\phantom{s}};
    \node (11) at (0,2) {s};
    \node (12) at (1,2) {s};
    \node (13) at (2,2) {s};
    \node (14) at (3,2) {s};
    \node (15) at (4,2) {s};
    \node (16) at (0,3) {\phantom{s}};
    \node (17) at (1,3) {\phantom{s}};
    \node (18) at (2,3) {\phantom{s}};
    \node (19) at (3,3) {\phantom{s}};
    \node (20) at (4,3) {\phantom{s}};
    \node (21) at (0,4) {\phantom{s}};
    \node (22) at (1,4) {\phantom{s}};
    \node (23) at (2,4) {\phantom{s}};
    \node (24) at (3,4) {\phantom{s}};
    \node (25) at (4,4) {\phantom{s}};
  \end{scope}

\begin{scope}[every path/.style={thick}]
    \path [-] (1) edge node {} (2);
    \path [-] (2) edge node {} (3);
    \path [-] (3) edge node {} (4);
    \path [-] (4) edge node {} (5);
    \path [-] (6) edge node {} (7);
    \path [-] (7) edge node {} (8);
    \path [-] (8) edge node {} (9);
    \path [-] (9) edge node {} (10);
    \path[red] [->] (11) edge node {} (16);
    \path[red] [->] (12) edge node {} (17);
    \path[red] [->] (13) edge node {} (18);
    \path[red] [->] (14) edge node {} (19); 
    \path[red] [->] (15) edge node {} (20);
    \path[red] [->] (16) edge node {} (21);
    \path[red] [->] (17) edge node {} (22);
    \path[red] [->] (18) edge node {} (23);
    \path[red] [->] (19) edge node {} (24); 
    \path[red] [->] (20) edge node {} (25);
    \path [-] (11) edge node {} (12);
    \path [-] (12) edge node {} (13);
    \path [-] (13) edge node {} (14);   
    \path [-] (14) edge node {} (15);
    \path [-] (11) edge node {} (16);    
    \path [-] (12) edge node {} (17);    
    \path [-] (13) edge node {} (18);    
    \path [-] (14) edge node {} (19);    
    \path [-] (15) edge node {} (20);  
    \path[blue] [->] (1) edge node {} (2); 
    \path[blue] [->] (6) edge node {} (7);  
    \path [-] (18) edge node {} (19);    
    \path  [-] (19) edge node {} (20);
    \path  [-] (16) edge node {} (17);
    \path  [-] (17) edge node {} (18);
    \path [-] (16) edge node {} (21);    
    \path [-] (17) edge node {} (22);
    \path [-] (18) edge node {} (23); 
    \path [-] (19) edge node {} (24);    
    \path [-] (20) edge node {} (25);
    \path [-] (21) edge node {} (22);
    \path [-] (22) edge node {} (23);
    \path [-] (23) edge node {} (24);
    \path [-] (24) edge node {} (25);
      \path [-] (1) edge node {} (6);
       \path [-] (2) edge node {} (7);
        \path [-] (3) edge node {} (8);
         \path [-] (4) edge node {} (9);
          \path [-] (5) edge node {} (10);
            \path [-] (11) edge node {} (6);
       \path [-] (12) edge node {} (7);
        \path [-] (13) edge node {} (8);
         \path [-] (14) edge node {} (9);
          \path [-] (15) edge node {} (10);
    \path[purple] [->] (3) edge node {} (4);
    \path[purple] [->] (4) edge node {} (5);
    \path[purple] [->] (8) edge node {} (9);
    \path[purple] [->] (9) edge node {} (10);
    \end{scope}
\end{tikzpicture}
\end{center}
\caption{Successful initial layout for $P_5 \Box P_5$ when $k=2$} 
\label{fig:Cartesian Product}
 
\end{figure}

\begin{corollary}
    If $r_1 =1$, $r_2 = 1$, $r_1 = k$ or $r_2 = k$, then $d_k(P_m \Box P_n) = \left \lceil \frac{mn}{k+1} \right \rceil$.
\end{corollary}

\begin{proof}

Since  $\left \lceil \frac{mn}{k+1} \right \rceil = \frac{mn-r_1r_2}{k+1}+ \left \lceil \frac{r_1r_2}{k+1}\right \rceil$ and   $d_k(P_m \Box P_n) \ge \left \lceil \frac{mn}{k+1} \right \rceil$, it suffices to show that 
$\left \lceil \frac{r_1r_2}{k+1}\right \rceil = \min\{r_1, r_2\}$.  

We have previously established the result when $r_1 = 0$ or $r_2 = 0$.  Therefore, 
without loss of generality, we assume $r_1 \ge 1$ and $r_2 \ge 1$. 

 If $r_1 = 1$ or $r_2=1$, then $\left \lceil \frac{r_1r_2}{k+1} \right \rceil =  1 = \min\{r_1, r_2\}$.  

 If $r_1 = k$, then $\left \lceil \frac{r_1r_2}{k+1} \right \rceil =  \left \lceil\frac{kr_2}{k+1}\right \rceil = \left \lceil r_2 - \frac{r_2}{k+1}\right \rceil$.  Since $\frac{r_2}{k+1} \le \frac{k}{k+1}<1$, it follows that $\left \lceil \frac{kr_2}{k+1}\right \rceil = r_2$.   Hence, $\left \lceil \frac{r_1r_2}{k+1} \right \rceil = r_2 = \min\{r_1, r_2\}.$  Similarly, if $r_2=k$, then $\lceil \frac{r_1}{r_2}{k+1}\rceil = r_1 = \min\{r_1,r_2\}$.
 \end{proof}

\bigskip
\begin{lemma}\label{lem:halfway1}
     Suppose $G= P_{k+1+r} \Box P_{s}$ such that $2 \le r < k$ and $\left \lceil \frac{r}2 \right \rceil< s$.  
     Let $\ell = \max \left\{ \left\lceil \frac{r}{2} \right\rceil ,\left \lceil \frac{2s + r - (k+1)}{2} \right\rceil \right \}$.
     If $\ell < r$, then $d_k(G) \le s +\ell$.  
\end{lemma}

\begin{proof}
  Since $r <k$, it follows that $\left \lceil \frac{2s + r - (k+1)}{2} \right\rceil \le  
\left \lceil \frac{2s -2}{2} \right\rceil <s$.  Since $\left \lceil \frac{r}2 \right \rceil< s$, it follows that $\ell < s$.

Begin by placing a searcher on each vertex of $S \cup S' \cup T$, where
    \begin{itemize}
        \item $S = \{(i,0): i=0, \ldots ,\ell -1\}$
        \item $S'= \{(k+1-s + \ell, j): j=s-\ell , \ldots , s - 1\}$
        \item $T = \{(\ell, j):j=0, \ldots , s-\ell-1\}$
    \end{itemize}
  
The initial placement of searchers and the paths they traverse are demonstrated in Figure \ref{secondL}.   Since $\ell < s \le k$, it follows that each of $S$, $S'$, and $T$ is non-empty, and a total of $s + \ell$ searchers are required.  

We see that in the initial stage, only searchers in $S \cup T$ move.  After the $(r-1)^{\mathrm{st}}$ stage, the searchers who started in $S$ occupy $(0,r-1), (1, r-2), \ldots (\ell-1, r-\ell)$, respectively.  This position is highlighted on the diagram as a dashed line extending from the top left corner. 

At stage $r$, only the searcher on $(0,r-1)$ moves. The searchers on $(1, r-2), \ldots (\ell-1, r-\ell)$ move again at stages $r+1$, $r+2$, ..., $r+\ell+1$, respectively.   This maneuvering of the corner is highlighted in the figure by the dashed line in the top left of the diagram. 

We note that the searchers on $T$  move in turn as their neighbours to the left are protected, with the searcher starting at $(\ell, s-\ell-1)$ moving after the last searcher to its left has turned the corner.  

The searchers in the top $\ell$ rows similarly maneuver a ``corner" as they approach the set $S'$.  The searchers who start on $S$ each traverse a path of length $k$ and therefore end on vertices $(k+1 - r, s-1), (k+2 - r, s-1), \ldots (k+\ell -r, s-1)$, respectively. It follows that, $k-s+ \ell +1$ vertices along the top row are protected by searchers that started on $S$.  Since $k-s+ \ell +1 \ge k+1 - s + \frac{2s+ r-(k+1)}{2} = \frac{k+1+r}{2}$, at least  half of the vertices in the top $\ell$ rows are protected by searchers in $S$. 

While the searchers that start in $S$ and $T$ will each traverse a path of length $k$, the searchers that start in $S'$ do not all traverse paths of length $k$, unless $2\ell = r$.  If $2\ell >r$, then the searchers from $S'$ in the top $2\ell -r$ rows will stop once they move onto the final column.   The remaining $r-\ell$ searchers will turn the corner and move down the last $r - \ell$ columns.
\end{proof}
\begin{figure}[ht]
\begin{center} 
\includegraphics[width = .7\linewidth]{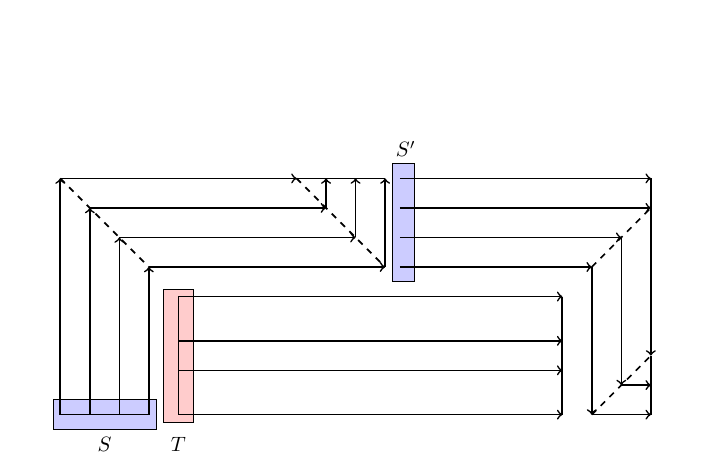}

\end{center}
\caption{An illustration of the searcher movements in the proof of Lemma~\ref{lem:halfway1}}\label{secondL}
\end{figure}
\begin{corollary}\label{cor:upper_bound_3}
    If $2 \le s, r \le k-1$, then $$d_k(P_{k+1+r} \Box P_{s}) \le s + \min\{s, r, \ell\}$$ where $\ell = \max \left\{ \left\lceil \frac{r}{2} \right\rceil ,\left \lceil \frac{2s + r - (k+1)}{2} \right\rceil \right \}$.
\end{corollary}

\begin{proof}
We know from Theorem \ref{Upper_bound_Cartesian2} that $d_k (P_{k+1+r} \Box P_{s}) \le s + \min\{s, r\}$.  

If $\left \lceil \frac{r}2 \right \rceil\ge s$, then $\ell \ge s$  and $\min\{\ell, s, r\} = \min \{s,r\}$.  % It follows that $d_k (P_{k+1+r} \Box P_{s}) \le s + \min\{s, r, \ell\}$.  
Similarly, if $ \ell \ge r$, then $\min\{\ell, s, r\} = \min \{s,r\}$.  Therefore, if $\left \lceil \frac{r}2 \right \rceil\ge s$ or $ \ell \ge r$, then  $d_k (P_{k+1+r} \Box P_{s}) \le s + \min\{s, r, \ell\}$.  

If $\left \lceil \frac{r}2 \right \rceil< s$ and $ \ell < r$, then by Lemma \ref{lem:halfway1}, $d_k(P_{k+1+r}\Box P_s) \le s + \ell$, and hence $d_k (P_{k+1+r} \Box P_{s}) \le s + \min\{s, r, \ell\}$.  
\end{proof}

\bigskip
\begin{corollary}
   Suppose $G= P_{k+1+r} \Box P_{s}$ where $2 \le r \le k-1$  and $s = \left \lceil \frac{k+1}{2} \right \rceil$.  Then $$d_k(G) = s +\left \lceil \frac{sr}{k+1} \right \rceil . $$  
\end{corollary}

\begin{proof}
From Lemma \ref{lem:gen_bound_1}, we have $d_k(G) \ge \left \lceil \frac{(k+1+r)s}{k+1} \right \rceil = s+ \left \lceil \frac{rs}{k+1} \right \rceil$.

Suppose $s = \left \lceil\frac{k+1}{2} \right \rceil$.  If $k+1$  is even, then $s = \frac{k+1}{2}$ and it follows from Corollary \ref{cor:upper_bound_3} that $d_k(G) \le s+ \min\{s, r, \left \lceil \frac r2 \right\rceil \} = s + \left \lceil \frac{r}{2}\right \rceil = s +\frac{sr}{k+1}$.  

If $k+1$ is odd, then $s = \frac{k+2}{2}$ and it  follows from Corollary \ref{cor:upper_bound_3} that $d_k(G) \le  s + \left \lceil \frac{r+1}{2}\right \rceil$.  Furthermore, $\left \lceil \frac{rs}{k+1} \right \rceil= \left \lceil \frac{r}{2} + \frac{r}{2(k+1)}\right \rceil$. Since $0<\frac{r}{2(k+1)}<\frac{1}{2}$,  it follows that $\left \lceil\frac{r}{2} + \frac{r}{2(k+1)}\right \rceil = \left \lceil\frac{r+1}{2} \right \rceil$.   Hence, $d_k(G) \le  s + \left \lceil \frac{rs}{k+1}\right \rceil$.

In either case, $d_k(P_{k+1+r} \Box P_s) = s+ \left \lceil\frac{sr}{k+1}\right \rceil$.     
\end{proof}
\bigskip

\begin{lemma}\label{lem: multi upper bound}
    Suppose $2 \le r, s \le k-1$.  Then $$d_k(P_{k+1+r} \Box P_s) \le s+ \ell + \min \left \{r-\ell, s-\ell, \max \left\{ \left\lceil \frac{r-\ell}{2} \right\rceil ,\left \lceil \frac{2s + r -3\ell - (k+1)}{2} \right\rceil \right \}\right \}$$ where $\max\{1, r+s - (k+1)\} \le \ell \le \min\{r,s\}$.

\end{lemma}

\begin{proof}
      
We use a similar technique as in Lemma \ref{lem:halfway1}, by placing one searcher on each vertex of $S \cup S'$ where $S = \{(i,0): i = 0, \ldots,  \ell -1\}$ and $S' = \{(k+\ell -s+1, j): j=s-\ell, \ldots , s-1\}$.  This is demonstrated in Figure \ref{thirdL}. This is a total of $2\ell$ searchers on $S \cup S$  the first $\ell$ columns and the top $\ell$ rows.   We note that the searchers in $S$ will search traverse a path of length $k$.  To guarantee that the searchers in $S'$ can protect the remaining vertices in their row, $\ell \ge s+r - (k+1)$.      

    We note that by removing the first $\ell$ columns and the top $\ell$ rows of $G$, we obtain a subgraph isomorphic to $P_{k+1+r - \ell} \Box P_{s-\ell}$ which, by Lemma \ref{lem:halfway1},  can be protected using $$s'+ \min \left \{r', s', \max \left\{ \left\lceil \frac{r'}{2} \right\rceil ,\left \lceil \frac{2s' + r'  - (k+1)}{2} \right\rceil \right \}\right \}$$ searchers, where $r'=r-\ell$ and $s'=s-\ell$. This subgraph is labeled as $H$ in Figure  \ref{thirdL} and the searchers in $H$ are placed in one of three ways. The first two are demonstrated in Figure \ref{H12} and the third is the layout in Figure \ref{secondL}.  In all cases, the searchers, when placed in similar positions on the subgraph $H$ in $G$, along with $2 \ell$ searchers on $S \cup S'$ move along the same paths as they would if $H$ were a stand-alone graph.  Therefore, the graph is protected with a total of $$2\ell + s'+ \min \left \{r', s', \max \left\{ \left\lceil \frac{r'}{2} \right\rceil ,\left \lceil \frac{2s' + r'  - (k+1)}{2} \right\rceil \right \}\right \}$$ searchers, where $r'=r-\ell$ and $s'=s-\ell$, and the result follows.
\end{proof}

\begin{figure}[ht]
\begin{center} 
\includegraphics[width = .7\linewidth]{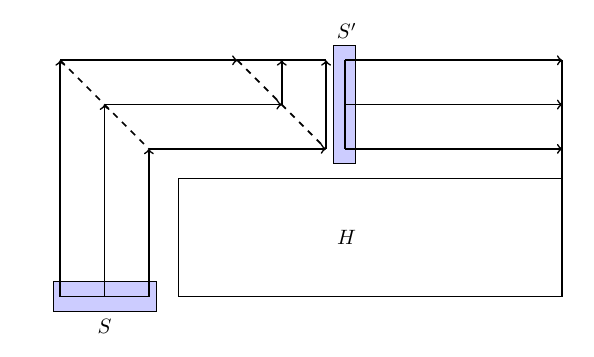}

\end{center}
\caption{An illustration of the searcher movements in the proof of Lemma~\ref{lem: multi upper bound}}\label{thirdL}
\end{figure}

\begin{figure}[h!t]
\begin{center} 
\includegraphics[width = .7\linewidth]{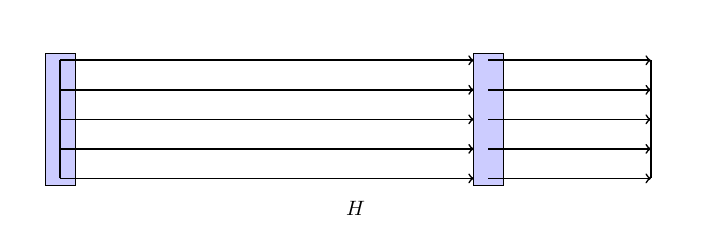}
\includegraphics[width = .7\linewidth]{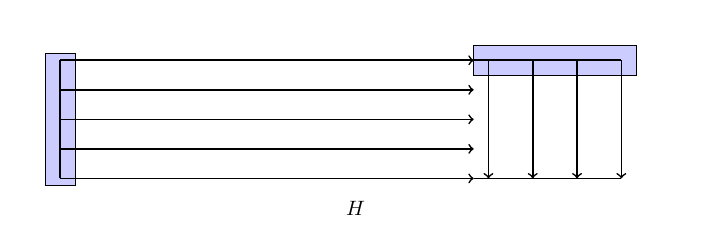}
\end{center}
\caption{Two of the starting configurations for searchers in the subgraph $H$ pictured in Figure~\ref{thirdL}}\label{H12}
\end{figure}

To get an idea as to how well the various upper and lower bounds presented for the Cartesian products of paths actually perform, we consider $P_{k+1+r} \Box P_{s}$ where $k \le 7$ and $0 \le s, r \le k$.

\begin{theorem}
   For all $k$, $r$ and $s$ such that $2 \le k \le 7$, $0 \le r \le k$, and $2 \le s \le k$, the value of $d_k(P_{k+1+r} \Box P_{s})$ can be determined exactly.   Furthermore, when $2 \le k \le 4$, $d_k(P_{k+1+r} \Box P_{s}) = \left \lceil \frac{(k+1+r)s}{k+1} \right \rceil= s + \left  \lceil\frac{rs}{k+1} \right \rceil$. 
\end{theorem}

\begin{proof}
Suppose $2 \le k \le 7$, $0 \le r \le k$, and $2 \le s \le k$.    By Corollaries 4.4 and 4.10, we know $d_k(P_{k+1+r}\Box P_{s}) = \left \lceil \frac{(k+1+r)s}{k+1} \right \rceil$ whenever $s =k$ or $r\in \{0, 1, k\}$.

By Lemma \ref{lem:gen_bound_1}, we have $d_k(P_{k+1+r} \Box P_{s}) \ge \left \lceil \frac{(k+1+r)s}{k+1} \right \rceil= s + \left  \lceil\frac{rs}{k+1} \right \rceil$.  In addition, by Corollary \ref{cor:improve_lower}, $d_k(P_{k+1+r} \Box P_{s}) \ge \min \{r-s+1, r\}$ whenever $r \ge s$.   This can be summarized as follows: $$  d_k(P_{k+1+r} \Box P_{s}) \ge  
   \begin{cases}
        \max\left \{s+\left\lceil\frac{rs}{k+1}\right\rceil,  s+\min \{r-s+1,r\} \right \} , & \text{if } r \ge s \\
        s+\left \lceil\frac{rs}{k+1} \right \rceil , & \text{if } r < s.
    \end{cases}$$
We will refer to this as the combined lower bound.

For $2 \le k \le 7$ and $2 \le r,s\le k-1$, we calculated all upper bounds on $d_k(P_{k+1+r} \Box P_{s})$ arising from results in this section (Lemma 4.7, Theorem 4.9, Corollary 4.12, Corollary 4.13 and Lemma 4.14).  In each case, we found an upper bound that was equal to the combined lower bound.   Therefore, in each case, an optimum initial layout can be found using the constructions discussed in this section.

Furthermore, when $2 \le k \le 4$ and $2 \le r,s\le k-1$, it follows that $k = 3$ or $k=4$. In each case, we can verify that $s+\left\lceil\frac{rs}{k+1}\right\rceil =  s+\min \{r-s+1,r\}$ for all $r$ and $s$ such that $2 \le s \le r \le k-1$.  \end{proof}

\begin{corollary}
 Suppose $2 \le k \le 4$.  If  $m \ge 2$ and $n \ge k+1$, then $$ d_k(P_m \Box P_n) = \left \lceil \frac{mn}{k+1} \right \rceil.$$   
\end{corollary}

\section{Strong Products}

We define the {\em strong product} of graph $G$ and $H$, denoted $G \boxtimes H$, as follows: $V(G \boxtimes H ) = V(G \Box H)$
 and $E(G \boxtimes H ) = E(G \Box H) \cup \{(x,y)(u,v): xu \in E(G) \mbox{ and } yv \in E(H)\}$. 
  In this section, we begin by finding a  lower bound on  $d_k(G \boxtimes H)$ for any connected graphs $G$ and $H$. Determining general non-trivial upper bounds proved very challenging. Such bounds would need to hold for $K_m\boxtimes K_n$ where nearly saturating the graph with searchers would be necessary. However, since this is likely to be very far off for less dense graphs, we would need further restrictions on the graph structure to obtain more useful upper bounds. We leave this as a possible future direction for the interested reader and instead focus on general lower bounds.  In the specific case that $G$ and $H$ are both paths, we give  upper bounds on $d_k(G \boxtimes H)$.

 \subsection{Lower Bounds on $d_k(G\boxtimes H)$}
 
 \begin{theorem} \label{thm:StrongLower}
  Suppose $2 \le m \le n$, and  $G$ and $H$ are connected graphs of orders $m$ and $n$, respectively.  For any $k \ge 1$,  $d_k(G \boxtimes H) \ge \min \{ m+ 2 \delta (G) -1, n+ 2 \delta (H) -1\}$. 
\end{theorem}

\begin{proof}
   Suppose $V(G) = \{x_1, x_2, \ldots, x_m\}$ and $V(H) = \{y_1, \ldots , y_n\}$.  For each $j = 1, \ldots, n$, let $G_j$ be the subgraph of $G \boxtimes H$ induced on $\{(x, y_j): x\in V(G)\}$.  Similarly, for each $i = 1, \ldots, m$, let $H_i$ be the subgraph induced on $\{(x_i, y): y \in V(H)\}$.  Let ${\cal G} = \{G_j: j=1, \ldots, n\}$ and ${\cal H} = \{H_i: i=1, \ldots, m\}$.  

Consider playing the $k$-move deduction game, where $k \ge 2$, on $G \boxtimes H$, with an initial successful layout using $d_k(G \boxtimes H)$ searchers.  We proceed in cases depending on the structure of this layout.

{\it Case 1: In the initial layout, there is at least one searcher on each copy of $H$ in ${\cal H}$, or at least one searcher on each copy of $G$ in ${\cal G}$.}   

Suppose that, in the initial layout, there is at least one searcher on each copy of $H$, and at stage 1, searchers move off vertex $v=(x,y)$.  It follows that in the initial layout there are at least $|N(v)|$ searchers occupying the set $N[v]$.   Furthermore, there are $m-\deg_G (x)-1$  copies of $H$ in $\cal H$ that do not contain any vertex from $N[v]$ and are each occupied by a searcher.
 It follows that there are at least $\deg_{G\boxtimes H}  (v) + (m-\deg_G(x)-1)$ searchers  in $G \boxtimes H$.  
 
 Since $\deg_{G\boxtimes H} (v) = (\deg_G (x) +1)(\deg_H (y) +1)-1 \ge (\delta(G) +1)(\delta(H) +1)- 1$ and 
 $\deg_G(x) \geq \delta(G)$, it follows that \[\deg_{G \boxtimes H}(v) + (m-\deg_G(x)-1) \geq (\delta(G)+1)(\delta(H)+1)-1+m-(\delta(G)+1).\]Therefore, $d_k(G \Box H) \ge \delta (H) (\delta (G) +1) + m-1 \ge 2 \delta (H) +m -1$.

Similarly, if there is at least one searcher on each copy of $G$ in the initial layout, then $d_k(G \boxtimes H) \ge m+ 2\delta (H) -1$.   Hence, in Case 1, we have $d_k(G \boxtimes H) \ge \min \{ n+ 2 \delta (G) -1, m+ 2 \delta (H) -1\}$.

{\it Case 2: In the initial layout, there is some copy of $G$ in ${\cal G}$ and some copy of $H$ in ${\cal H}$ such that neither is occupied by a searcher.}

As in the proof of Theorem \ref{lower_bound_Cartesian}, 
let ${\cal H}_p$ be the subset of ${\cal H}$ consisting of copies of $H$ that are partially or fully protected at stage $t$.  Define ${\cal G}_p$ similarly.   Furthermore, $t$ is chosen such that ${\cal G}_p, {\cal H}_p,{\cal G} - {\cal G}_p$ and ${\cal H} - {\cal H}_p$ are all non-empty, but at stage $t+1$ every $G \in {\cal G}$ and $H \in {\cal H}$ is at least partially protected.  

It follows, as in Theorem \ref{lower_bound_Cartesian} that every $G \in {\cal G}_p$ and every $H \in {\cal H}_p$ is partially protected.

Without loss of generality, suppose $H_1 \in {\cal H}_p$ and $H_2 \in {\cal H} - {\cal H}_p$, and a searcher moves from  a vertex $v$ in $H_1$ to a vertex $w$ in $H_2$ at stage $t+1$.  Without loss of generality, suppose $v = (x_1,y_1)$ {and $w=(x,y_2)$ for some $x \in \{x_1, \ldots, x_m\}$}.  Note that, at stage $t$, every protected vertex in $H_1$ is a boundary vertex since 
$(x_i,y_1)$ is adjacent to $(x_i,y_2)$ for all $i$ and $H_2$ is totally unprotected.  Hence, every protected vertex in $H_1$ must be occupied by a searcher at the end of round $t$.  Therefore, $N[v] \cap V(H_1)$ contains at least $2|N_H [v]|-1 \ge 2\delta (H) +1$ searchers.   

Now, suppose there is another unprotected copy of $H$, say $H_3$.  If $x_1$ is adjacent to $x_3$ in $H$, it follows that an additional $\delta (H) +1$ searchers must appear in $H_1$.    If $x_1$ is not adjacent to $x_3$ in $H$, then searchers move from another vertex $v'$ in some $H' \in {\cal H}_p$ onto $H_3$ and there must be at least  $2\delta (H) +1$ searchers on $H'$.

We therefore have the following at the end of round $t$:
\begin{enumerate}
    \item  each copy of $H$ in ${\cal H}_p$ is occupied by at least one searcher,
    \item for the first unprotected copy of $H$ in $\cal H$, there must be least $2\delta (H)$ additional searchers,
    \item for each subsequent unprotected copy of $H$ in $\cal H$, there must be at least $\delta (H) +1$ additional searchers.
\end{enumerate}  

It follows that there are at least $|{\cal H}_p| + 2\delta(H) + (m-|{\cal H}_p|-1)(\delta (H)+1) \ge m+ 2 \delta (H) -1$ searchers on $G \boxtimes H$.  Hence, $d_k(G \boxtimes H) \ge  m+ 2 \delta (H) -1\ge \min\{ n+ 2 \delta (G) -1, m+ 2 \delta (H) -1\}$.
\end{proof}

\subsection{Strong Product of Paths}
 Label the vertices of the $m \times n$ strong grid using the set $\{(i,j) \mid 0 \le i \le m-1, 0 \le j \le n-1\}$, where $(i,j)$ and $(i',j')$ are adjacent if $|i-i'|\le 1$ and $|j-j'| \le 1$.  Rows $R_j$ and columns $R_i$ are defined as they were for  Cartesian grids.

\begin{theorem}
\label{lemma:strong_product_upper}
Suppose $k\ge 1$ and $3 \le m \le n \le k+1$.  Then $d_{k}(P_m \boxtimes P_{n}) = %\le 
m+1$.
\end{theorem}

\begin{proof}
The lower bound follows from Theorem~\ref{thm:StrongLower}.  For the upper bound, suppose we play on  $P_m \boxtimes P_{n}$  with $m+1$ searchers where $k \ge 1$ and $3 \le m \le n \le k+1$, and each searcher can move at most $k$ times.  

Place two searchers on  $(0,0)$.  In addition, place a single searcher on each vertex in the set $\{(i, 0) \mid 1\le i \le m-1\}$, respectively. The searchers on $(0,0)$ can move in the first time-step -- one moves to $(0,1)$ while the other moves to $(1,1)$.  Now, for each $i=1, \ldots , {m-3}$ the searcher on $(i,0)$ moves, in turn, to $(i+1,1)$.  Following this, the searchers on {$(m-2,0)$ and} $(m-1,0)$ move to $(m-1,1)$.  (See  Figure~\ref{fig:StrongPath}.)

After one move each, the searchers  now occupy all vertices of the next row with two searchers on the vertex $(m-1,1)$.  Since the layout of searchers is the mirror image of the initial layout shifted up one row, we can see that $m+1$ searchers can subsequently clear up to $k+1$ rows using at most $k$ moves each.
\end{proof}

\begin{center}
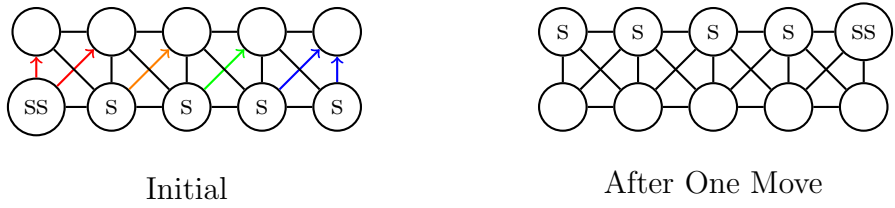
\begin{figure}[ht]
\centering
\begin{tikzpicture}
\begin{scope}[every node/.style={circle,thick,draw}]
    \node (1) at (0,0) {ss};
    \node (2) at (1,0) {s};
    \node (3) at (2,0) [label=below: Initial] {s};   
    \node (4) at (3,0) {s};   
    \node (5) at (4,0) {s};
    \node (6) at (0,1) {\phantom{s}};
    \node (7) at (1,1) {\phantom{s}};
    \node (8) at (2,1) {\phantom{s}};
    \node (9) at (3,1) {\phantom{s}};
    \node (10) at (4,1) {\phantom{s}};
  \end{scope}

\begin{scope}[every path/.style={thick}]
    \path [-] (1) edge node {} (2);
    \path [-] (2) edge node {} (3);
    \path [-] (3) edge node {} (4);
    \path [-] (4) edge node {} (5);
    \path [-] (6) edge node {} (7);
    \path [-] (7) edge node {} (8);
    \path [-] (8) edge node {} (9);
    \path [-] (9) edge node {} (10);
    \path[red] [->] (1) edge node {} (6);
    \path [-] (2) edge node {} (7);
    \path [-] (3) edge node {} (8);    \path [-] (4) edge node {} (9); 
    \path[blue] [->] (5) edge node {} (10);
    \path[red] [->] (1) edge node {} (7);
    \path[orange] [->] (2) edge node {} (8);
    \path[green] [->] (3) edge node {} (9);    
    \path[blue] [->] (4) edge node {} (10); 
    \path [-] (2) edge node {} (6);
    \path [-] (3) edge node {} (7);    \path [-] (4) edge node {} (8);  
    \path [-] (5) edge node {} (9);    
\end{scope}

\begin{scope}[every node/.style={circle,thick,draw}]
    \node (11) at (7,0) {\phantom{s}};
    \node (12) at (8,0) {\phantom{s}};
    \node (13) at (9,0) {\phantom{s}};   
    \node (14) at (10,0) {\phantom{s}};   
    \node (15) at (11,0) {\phantom{s}};
    \node (16) at (7,1) {s};
    \node (17) at (8,1) {s};
    \node (18) at (9,1) {s};
    \node (19) at (10,1) {s};
    \node (20) at (11,1) {ss};
  \end{scope}

\draw (9,-1) node{After One Move};

\begin{scope}[every path/.style={thick}]
    \path [-] (11) edge node {} (12);
    \path [-] (12) edge node {} (13);
    \path [-] (13) edge node {} (14);
    \path [-] (14) edge node {} (15);
    \path [-] (16) edge node {} (17);
    \path [-] (17) edge node {} (18);
    \path [-] (18) edge node {} (19);
    \path [-] (19) edge node {} (20);
    \path [-] (11) edge node {} (16);
    \path [-] (12) edge node {} (17);
    \path [-] (13) edge node {} (18);   \path [-] (14) edge node {} (19); 
    \path [-] (15) edge node {} (20);
    \path [-] (11) edge node {} (17);
    \path [-] (12) edge node {} (18);
    \path [-] (13) edge node {} (19);    
    \path [-] (14) edge node {} (20); 
    \path [-] (12) edge node {} (16);
    \path [-] (13) edge node {} (17);    \path [-] (14) edge node {} (18);  
    \path [-] (15) edge node {} (19);  
\end{scope}
\end{tikzpicture}
\caption{Strong products of paths} \label{fig:StrongPath}
\end{figure}
\end{center}

We note that the initial layout of searchers described above is not successful when $m=2$.

\begin{theorem}\label{lem:smallstrong} Suppose $k \ge 2$. Then 
    $$d_k(P_2 \boxtimes P_n) =   
    \begin{cases}
        3, & \text{if } 2\le n \le 3 \\
        4, & \text{if } 4 \le n \le k+1. 
    \end{cases}$$
   
\end{theorem}

\begin{proof}
   By Lemma~\ref{lem:gen_bound_1}, we have $d_k(P_2 \boxtimes P_n) \ge 3$ for  all $n \ge 3$.
    Since $P_2 \boxtimes P_2 \cong K_4$, by Lemma~\ref{lem:path_cycle_complete}, $d_k(P_2 \boxtimes P_2) = 3$.

    For $P_2 \boxtimes P_n$, where $n \ge 3$, consider an initial layout with three searchers where each searcher occupies a unique vertex. {Note that in any such layout, a vertex of degree $3$ and two of its neighbours must be occupied; otherwise, no searchers can move.} Without loss of generality, assume there is a searcher on each of $(0,0)$, $(1,0)$, and $(0,1)$.  In round 1,  the first two searchers move to $(1,1)$ while the third searchers stays on $(0,1)$.  Then in round 2, the searchers on $(1,1)$ move onto $(0,2)$ and $(1,2)$.  If $n=3$, then three searchers have cleared all vertices.  However, if $n \ge 4$, no searcher can move in round 3.  Therefore, $d_k(P_2 \boxtimes P_3) = 3$ and there is no successful initial layout using three searchers on three distinct vertices of $P_2 \boxtimes P_n$ when $n \ge 4$. 

    Now consider an initial layout of three searchers on $P_2 \boxtimes P_n$, $n \ge 4$, where at least two searchers share a vertex.  It follows that, without loss of generality, there are either three searchers on $(0,0)$, or two searchers on $(0,0)$ and one on {either $(0,1)$ or} $(1,0)$.  In either case, after the first round, there will be one searcher on each of $(1,0)$, $(0,1)$, and $(1,1)$.  This means that no searcher can move in round 2.   It follows that there is no successful initial layout on $P_2 \boxtimes P_n$, where $n \ge 4$, using three searchers.

    Finally, consider an initial layout using four searchers where two searchers occupy $(0,0)$ and two searchers occupy $(1,0)$.   All searchers move in round 1, and the round concludes with two searchers on $(0,1)$ and two searchers on $(1,1)$; {see Figure~\ref{fig:StrongPath2}}.  We see, inductively, that for each $t = 1, \ldots, n-1$, two searchers occupy $(0,t)$ and two searchers occupy $(1,t)$ at the end of round $t$.  Hence, this is a successful initial layout.   
\end{proof}

\begin{center}
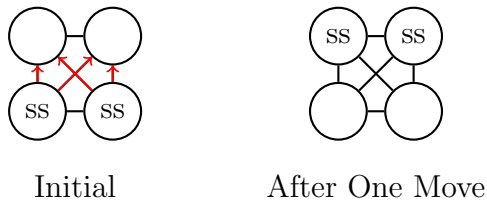
\begin{figure}[ht]
\centering
\begin{tikzpicture}
\begin{scope}[every node/.style={circle,thick,draw}]
    \node (1) at (0,0) {ss};
    \node (2) at (1,0)  {ss};   
    \node (3) at (0,1) {\phantom{ss}};
    \node (4) at (1,1) {\phantom{ss}};
    \node (5) at (4,0) {\phantom{ss}};
    \node (6) at (5,0) {\phantom{ss}}; 
    \node (7) at (4,1) {ss};
    \node (8) at (5,1) {ss};
  \end{scope}

  \draw (0.5,-1) node{Initial};
  \draw (4.5,-1) node{After One Move};

\begin{scope}[every path/.style={thick}]
    \path [-] (1) edge node {} (2);
    \path [-] (1) edge node {} (3);
    \path [-] (1) edge node {} (4);
    \path [-] (2) edge node {} (3);
    \path [-] (2) edge node {} (4);
   \path [-] (3) edge node {} (4);
    \path[red] [->] (1) edge node {} (3);
    \path[red] [->] (1) edge node {} (4);
     \path[red] [->] (2) edge node {} (3);
    \path[red] [->] (2) edge node {} (4);
      
     \path [-] (5) edge node {} (6);
    \path [-] (5) edge node {} (7);
    \path [-] (5) edge node {} (8);
    \path [-] (6) edge node {} (7);
    \path [-] (6) edge node {} (8);
    \path [-] (7) edge node {} (8);
   
\end{scope}

\end{tikzpicture}
\caption{The initial layout and first move for $P_2 \boxtimes P_n$}
\label{fig:StrongPath2}
\end{figure}
\end{center}

\begin{lemma} \label{lem:StrongUpperBound1}
   Suppose $2 \le m \le n$ and $n > k+1$.  Then 
   $$d_k(P_m \boxtimes P_n) \le
   \begin{cases}
      4\left \lceil \frac{n}{k+1} \right \rceil &\text{ if } m = 2 \\ \\
     (m+1) \left \lceil \frac{n}{k+1} \right \rceil &\text{ if }   m \ge 3. 
   \end{cases}$$

\end{lemma}

\begin{proof}
Suppose $m \ge 3$.   We can place $m+1$ searchers on row $R_0$ of $P_m \boxtimes P_n$, as well as on each of the rows $R_j$ for all $j>0$ such that $j \equiv s \pmod{k+1}$.  For each of these rows, the searchers are placed exactly as they are in Figure \ref{fig:StrongPath}.  
    
    Note that, unlike with the graph $P_m \Box P_n$, there are issues of timing to address in $P_m \boxtimes P_n$.  Specifically, we need to verify that searchers moving onto the same vertex in a stand-alone copy of $P_m \boxtimes P_{k+1}$ will move simultaneously $P_m \boxtimes P_n$.   This is straightforward to confirm.   Hence,  $$d_k(P_m \boxtimes P_n) \le \left \lceil \frac{n}{k+1} \right \rceil (m+1).$$

If $m = 2$, then place two searchers on each vertex of row $R_0$.  Similarly place four searchers on  rows $R_j$ for all $j>0$ such that $j \equiv s \pmod{k+1}$.   Again, it is straightforward to confirm that the this is a successful initial layout, and the result follows. 
\end{proof}

\begin{lemma}\label{lemma:partitionstrong}
   Suppose $2 \le m \le n$,  $n > k+1$,  $n_1+ n_2 = n$ and $n_2 \equiv 0 \bmod{(k+1)}$.  If there is a successful initial layout on $P_{m} \boxtimes P_{n_1}$ using $\ell$ searchers such that the initial layout includes all vertices in either the top or the bottom row of $P_{m} \boxtimes P_{n_1}$ then $$d_k(P_m \boxtimes P_n) \le     \begin{cases}
       \ell +\frac{4n_2}{k+1}  &\text{ if } m = 2 \\ \\
       \ell + \frac{(m+1)n_2}{k+1} &\text{ if }   m \ge 3. 
   \end{cases} $$

   If $m \ge 3$ and there is a successful initial layout on $P_{m} \boxtimes P_{n_1}$ using $\ell$ searchers such that there is a corner of $P_{s} \boxtimes P_m$ that isn't protected until the last  stage of the game, then $$d_k(P_m \boxtimes P_n) \le \ell +  \frac{n_2(m+1)}{k+1}  .$$

\end{lemma}

\begin{proof}  Let $G = P_m \boxtimes P_n$.  
As in the proof of Lemma \ref{lemma:partition} we can partition the strong grid $G$ into subgraphs $P_m \boxtimes P_{n_1}$ and $P_m \boxtimes P_{n_2}$ where $n_2 \equiv 0 \bmod{(k+1)}$. 

Following from Lemma \ref{lem:StrongUpperBound1}, there is a successful initial layout on 
 $P_m \boxtimes P_{n_2}$ such that every vertex in its bottom row is occupied by a searcher.  Using this same layout on the top $n_2$ rows of $G$ means that searchers in the bottom $n_1$ rows will move exactly as they would in a  stand-alone graph $P_m \boxtimes P_{n_1}$ . However, protecting vertices the first $n_1$ rows can affect movement of the searchers in the rest of the graph due to timing issues.  Therefore, we need to be careful as to how searchers are placed in $P_m \boxtimes P_{n_1}$.

{First}, suppose there is an initial successful layout using $\ell$ searchers in $P_m \boxtimes P_{n_1}$ so that every vertex in the top row of $P_m \boxtimes P_{n_1}$ is occupied by a searcher.  By applying this same layout to the subgraph induced on rows $R_0$ through $R_{n_1-1}$,  together with the layout on the top $n_2$ rows described in the previous paragraph, the result is an initial successful layout in $G$.    

{Now} suppose there is an initial successful layout using $\ell$ searchers in $P_m \boxtimes P_{n_1}$ so the vertex $(0,n_1-1)$ isn't protected until the last  stage of the game.  This, together with the layout on the top $n_2$ rows described above, is an initial successful layout in $G$ since no searcher in row $R_{n_1}$ in $G$ moves until vertex $(0,n_1-1)$ is protected, or equivalently, until after every vertex in rows $R_0$ through $R_{n_1-1}$   is protected.
\end{proof}

\bigskip

\begin{theorem} \label{upper_bound_stronggrid}
   Suppose $k \ge 3$, $3 \le m \le n$,  and $k+1 < n$. Let $r= m \bmod {(k+1)}$ and $s = n \bmod {(k+1)}$.  Then $$d_k(P_m \boxtimes P_n) \le \begin{cases}  
(m+1)\frac{n}{k+1}  &\text{ if } s=0 \\
   (m+1)\left\lfloor \frac{n}{k+1} \right\rfloor +  \left\lceil \frac{m}{k+1}\right \rceil  & \text{ if } s = 1 \\
   (m+1)\left\lfloor \frac{n}{k+1} \right\rfloor +  4\left\lfloor \frac{m}{k+1}\right \rfloor +d_k(P_r \boxtimes P_s) & \text{ if } s= 2 \\

 (m+1)\left\lfloor \frac{n}{k+1} \right\rfloor + (s+1) \left\lfloor \frac{m}{k+1}\right \rfloor + d_k(P_r \boxtimes P_s) & \text{ if } s \ge 3 \\
   \end{cases}$$   

   where $d_k(P_0 \boxtimes P_s) = 0$.
\end{theorem}

\begin{proof}
When $s=0$, the result is a restatement of Lemma \ref{lem:StrongUpperBound1}.  So, we will assume that $s>0$.

By Lemma \ref{lemma:partitionstrong}, with $n_1 = s$ and $n_2 = n-s$,  we have $d_k(P_m \boxtimes P_n) \le \ell+\frac{(n-s)(m+1)}{k+1}=\ell+ (m+1)\left\lfloor \frac{n}{k+1}\right\rfloor$ if there is a successful initial layout in $P_m \boxtimes P_s$ (or $P_s \boxtimes P_m$) using $\ell$ searchers such that the last vertex of $P_m \boxtimes P_s$ protected is a corner.

{\it Case 1: $s=1$.}  Then $P_m \boxtimes P_s$ is the path $P_m$ and the required initial layout can be obtained with $\ell = \left \lceil \frac{m}{k+1} \right \rceil$ searchers by Lemma~\ref{lem:path_cycle_complete}.

{\it Case 2: $s=2$.}  Then we need a successful initial layout on $H = P_m \boxtimes P_2$ such that a corner vertex is protected in the final round.

If $r \ge 4$ or $r=0$, then there is a successful initial layout in $H$ in which four searchers are placed on $C_0$ {(two on each vertex)} and four searchers on each $C_j$ for $j >0$ and $j \equiv r \pmod{k+1}$.  When the game is played on $H$, column $C_{m-1}$ is protected in the last  stage.  Therefore, there is a corner that is not protected until the final  stage.  %{\color{gray}It follows that $\ell = 4 \left \lceil \frac{m}{k+1} \right \rceil =  4 \left \lfloor \frac{m}{k+1} \right \rfloor + 4 = 4 \left \lfloor \frac{m}{k+1} \right \rfloor + d_k(P_2 \boxtimes P_r) $as required.} 
{The number of searchers we have used on $H$ is $\ell = 4 \left \lceil \frac{m}{k+1} \right \rceil$.  If $r =0$, then $\ell=4 \left \lfloor \frac{m}{k+1} \right \rfloor$; otherwise $\ell= 4 \left \lfloor \frac{m}{k+1} \right \rfloor + 4$.}

If $r=3$, then place a searcher on  $(0,0), (1,0)$ and $(0,1)$ in $H$, as well as  %{\color{gray}\sout{four}} 
{two } searchers on {each vertex of} column $C_j$ for each $j >0$ such that $j \equiv r \pmod{k+1}$.  The searchers on  $(0,0), (1,0)$ and $(0,1)$ protect columns $C_0, C_1, C_2$ of $H$ (similar to the proof of Lemma \ref{lem:smallstrong}).  Furthermore, both vertices in $C_2$ are protected in the same  stage.  It follows that the remaining searchers then protect columns $C_3$ through $C_{m-1}$,  where both vertices in $C_{m-1}$ are protected in the final  stage.

If $r=2$, place a searcher on each of $(1,1)$, $(0,1)$ and $(1,0)$ and %{\color{gray}\sout{four}} 
{two} searchers on each {vertex of} $C_j$ for all $j>0$ and $j \equiv r \pmod{k+1}$.  (Note that $m>k+1$ since $m >r$.) It follows that his initial successful layout in $P_2 \boxtimes P_m$ results in both vertices in column $C_{m-1}$ protected in the final  stage.   

If $r=1$ then place a single searcher on $(0,0)$ and %{\color{gray}\sout{four}}
{two} searchers on each {vertex of $C_j$} %{\color{gray}$R_j$} 
for $j>0$ and $j \equiv r \pmod{k+1}$.   

{Recall by Theorem~\ref{lem:smallstrong} that $d_k(P_2 \boxtimes P_r)=4$ if $r \geq 4$, and $d_k(P_2 \boxtimes P_r)=3$ if $2 \leq r \leq 3$.  If $r=1$, then $d_K(P_2 \boxtimes P_r) = d_k(P_2)=1$.  Thus,}
for all values of $r$, $H$ is protected using $4 \left \lfloor \frac{m}{k+1} \right \rfloor + d_k(P_r \boxtimes P_2)$ searchers and a corner vertex is protected at the final  stage.

{\it Case 2: $s\ge 3$.} Let $H' = P_r \boxtimes P_s$. 
 If $r=1$, place one searcher on $(0,0)$.  If $r = 2$, then place two searchers on each of $(0,0)$ and $(1,0)$.   If $r \ge 3$, then place $\min\{r,s\} +1$ searchers on $H'$ as described in {Theorem~\ref{lemma:strong_product_upper}}.     
 For all values of $r$, $d(P_r \boxtimes P_s)$ searchers can protect $H'$ such that a corner vertex is protected in the final  stage.    It follows from Lemma \ref{lemma:partitionstrong}, that $P_m \boxtimes P_s$ can be protected using $(s+1) \left \lfloor \frac{m}{k+1} \right \rfloor + d_k(P_r \boxtimes P_s)$ searchers so that a corner is protected in the final  stage. 
   \end{proof}

\section{Discussion}

In this paper, we have extended the deduction game to the scenario in which searchers may move up to $k$ times.  We have given bounds on the $k$-move deduction number and examined the game on Cartesian and strong products.  We close by considering several open questions for future research.

The majority of the results presented in this paper focus on bounds for the product of paths. A natural next direction is to consider the Cartesian product of a path and a cycle. The structure is very similar to the Cartesian product of paths, so intuition would suggest that similar tools could be used in these cases. This is indeed true in cases such as in Lemma~\ref{lemma:partition}. The only additional edges are between two copies of outer paths. In the initial configuration of $P_n \square P_m$, the searchers placed in the corners now correspond to the searchers in $P_n \square  C_m$ which are adjacent. Thus, they still each only have one neighbour that is not occupied by a searcher. Thus, the same strategy can be applied, except possibly at the very end stage of searching. Where this gets challenging is that there are fewer boundary pieces to the graph, and thus new adjacencies can present problems where searchers could have moved in $P_n \Box P_m$ now become stalled. Worse is more complex configurations for upper bounds, as presented in results such as Lemma~\ref{lem: multi upper bound}, trying to adapt this configuration will not work because some of the searchers as placed, can not move in the first stage. Though the local structures are similar to the Cartesian product of paths, we do not get as much carryover as one would expect. This is a natural next direction for future work.

Another intriguing question is whether $d_k(G) \leq d_{\ell}(G)$ whenever $k \geq \ell$.  Clearly it is true that $d_k(G) \leq d(G)$ for all $k \geq 1$, since any layout for which the searchers protect all vertices after the first move is also successful in the $k$-move game.  However, in general, a successful layout in the $\ell$-move deduction game with $k \geq \ell$ is not necessarily successful in $k$-move deduction.  As an example, consider the graph $G$ in Figure~\ref{fig:timing}, for which the indicated layout of searchers is successful in $2$-move deduction, but not in $3$-move deduction.

\begin{center}
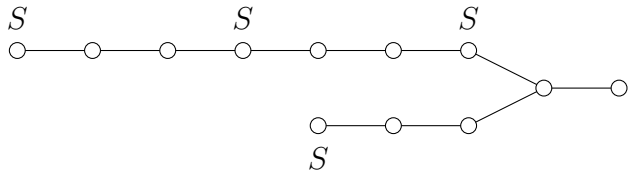
\begin{figure}[ht]
\centering
    \begin{tikzpicture}[x=1cm,y=1cm,scale=1]
    \draw(0,1) node[above=4pt]{$S$};
    \draw(3,1) node[above=4pt]{$S$};
    \draw(1,0) node[below=4pt]{$S$};
    \draw(-3,1) node[above=4pt]{$S$};
    \draw (-3,1) -- (3,1);
    \draw (1,0) -- (3,0);
    \draw (3,1) -- (4,0.5) -- (5,0.5);
    \draw (3,0) -- (4,0.5);
    \foreach \x in {-3,-2,-1,0,1,2,3}{
    \draw (\x,1)[fill=white] circle (3pt);
    }
    \foreach \x in {1,2,3}{
    \draw[fill=white] (\x,0) circle (3pt);
    }
    \foreach \x in {4,5}{
    \draw[fill=white] (\x,0.5) circle (3pt);
    }
    \end{tikzpicture}
\caption{A layout that is successful for $2$-move deduction but not $3$-move deduction} \label{fig:timing}
\end{figure}
\end{center}

We also note that some choices had to be made in determining the rules for $k$-move deduction, so alternative games could be formulated by a slight change in rule.  In particular, consider the scenario where there are more searchers on a vertex than unprotected neighbours of that vertex.  In the $1$-move game, excess searchers cannot protect additional vertices, so that it is irrelevant if they are permitted to move or not.  We elected to allow excess searchers to move to any of the previously unprotected neighbours; requiring them to remain motionless would produce a new variant. 

Similarly, the $1$-move game is not affected if searchers are permitted to move to protected or occupied vertices.  While we did not permit searchers whose neighbourhood is completely protected to move, allowing this could potentially reduce the number of searchers required. 

Finally, we note that in~\cite{BDOWXY}, it was proved the in the $1$-move case, the number of searchers required to protect the graph does not change if it is not required that all searchers who may move at a given stage do so; this variant was termed {\em free deduction}.  In the $k$-move game, however, changing this rule may make a difference to the success of a layout. Figure~\ref{fig:timing} again provides an example to illustrate this; it is not successful in $3$-move deduction due to the timing of the searcher on the bottom path moving to the the vertex of degree $3$, but if we have the freedom to choose not to move this searcher to the vertex of degree $3$, it becomes a successful layout. Characterizing when the deduction number changes if searchers are no longer required to move when able would be an interesting question. 

\section{Statements and Declarations}
This work was supported by the Natural Sciences and Engineering Research Council of Canada (NSERC). Grant numbers: RGPIN-2025-04633 (Burgess), RGPIN-2020-06528 (Clarke), RGPIN-2023-03395 (Huggan) and DGECR-2023-00190 (Huggan).

\noindent
The authors have no competing interests to declare.

\section{Acknowledgments}
Burgess, Clarke and Huggan gratefully acknowledge Discovery Grant support from NSERC. Huggan further acknowledges support from an NSERC Discovery Launch Supplement.

\end{document}